\documentclass[oneside,a4paper,11pt,reqno]{amsart}
\usepackage{amssymb,amsmath,amsthm}
\usepackage{graphics,graphicx}

\newcommand{\Z}{{Z\!\!\!Z}}

\newtheorem{hypo}{Hypothesis}

\newtheorem{prop}[hypo]{Proposition}

\newtheorem{thm}[hypo]{Theorem}

\newtheorem{lem}[hypo]{Lemma}

\newtheorem{defi}[hypo]{Definition}

\newtheorem{nota}[hypo]{Notation}

\newtheorem{rqe}[hypo]{Remark}

\newtheorem{coro}[hypo]{Corollary}

\newtheorem{exa}[hypo]{Example}

\DeclareMathOperator{\mdeg}{mdeg}
\DeclareMathOperator{\val}{val}
\DeclareMathOperator{\base}{Base}
\DeclareMathOperator{\hess}{Hess}
\DeclareMathOperator{\Idd}{Id}

\def\B{\mathcal{B}}

\def\F{\mathcal{F}}
\def\Z{\mathcal{Z}}

\def\PP{\mathbb{P}}

\def\CC{\mathbb{C}}

\newcommand {\refeq}[1] {(\ref{#1})}

\title{On caustics by reflection of algebraic surfaces}
\date\today

\author{Alfrederic Josse}
\address{Universit\'e de Brest,
UMR CNRS 6205, LMBA,
6 avenue Le Gorgeu, 29238 Brest cedex, France}
\email{alfrederic.josse@univ-brest.fr}
\author{Fran\c{c}oise P\`ene}
\address{Universit\'e de Brest,
UMR CNRS 6205, LMBA,
6 avenue Le Gorgeu, 29238 Brest cedex, France}
\email{francoise.pene@univ-brest.fr}

\keywords{caustic, class, polar, intersection number, pro-branch\\
Fran\c{c}oise P\`ene is supported by the french ANR project GEODE (ANR-10-JCJC-0108)}
\begin{document}

\begin{abstract}
Given a point $ S$ (the light position)  in $\mathbb P^3$  and an algebraic surface $\mathcal Z$ (the mirror)  of $\mathbb P^3$, 
the caustic by reflection $\Sigma_{ S}(\mathcal Z)$ 
of $\mathcal Z$ from $ S$ is the Zariski closure of the 
envelope of the reflected lines $\mathcal R_m$ got by reflection of $(Sm)$ on $\mathcal Z$ at $m\in\mathcal Z$.
We use the ramification method to identify $\Sigma_{ S}(\mathcal Z)$ with the Zariski closure of the
image, by a rational map, of an algebraic 2-covering space of $\mathcal Z$. 
We also give a general formula for the degree (with multiplicity) of caustics (by reflection) of 
algebraic surfaces of $\mathbb P^3$.
\end{abstract}

\maketitle
\section*{Introduction}
Let $ S[x_0:y_0:z_0:t_0]\in\mathbb P^3:=\mathbb P(\mathbf{W})$
(with $\mathbf {W}$ a 4-dimensional complex vector space)
and let $\mathcal Z=V(F)$ be a surface of $\mathbb P^3$
given by some $F\in Sym^d(\mathbf{W}^\vee)$
(i.e. $F$ corresponds to a polynomial of degree $d$
in $\mathbb C[x,y,z,t]$).
The {\bf caustic by reflection} $\Sigma_{S}(\mathcal Z)$ of $\mathcal Z$ from $ S\in\mathbb P^3$
is the Zariski closure of the envelope of the reflected lines $\mathcal R_m$ of the
lines $(m S)$ after reflection at $m$ on the mirror surface $\mathcal Z$.

Since the seminal work of von Tschirnhaus \cite{Tschirnhaus1,Tschirnhaus2}, caustics by reflection of planar curves have been studied namely by Chasles \cite{Chasles}, Quetelet \cite{Quetelet} and Dandelin \cite{Dandelin}. Let us also mention the work of Bruce, Giblin and Gibson \cite{BG,BGG1,BGG2} in the real case.
A precise computation of the degree and class of caustics by reflection of planar
algebraic curves has been done in \cite{fredsoaz1,fredsoaz2,fredsoaz3}. 
The idea was based on the fact that the caustic by reflection of an irreducible algebraic curve 
$\mathcal C$ of $\mathbb P^2$ from source $ S_0\in\mathbb P^2$ is the Zariski closure of the image of $\mathcal C$
by a rational map. 
Moreover, in the planar case,
the generic birationality of the caustic map has been established
in \cite{fredsoaz3,Catanese}.
The study of caustics by reflection of algebraic surfaces
is more delicate. We will see that a generic point $m$ of  $\mathcal Z$ is associated to two (instead of a single one) 
points on $\Sigma_{ S}(\mathcal Z)$.

A classical way to study envelopes is the ramification theory. Let us mention that this approach has been used namely  by Trifogli in \cite{Trifogli}
and by Catanese and Trifogli in \cite{CataneseTrifogli} for focal loci (which generalize the notion of evolute to higher dimension).
We use here the ramification theory to construct the caustic by reflection $\Sigma_{ S}(\mathcal Z)$ and to identify it with 
the Zariski closure of the image by some rational map ${\Phi}$, of an algebraic 2-covering space $\hat Z$ of $\mathcal Z$. 
We will see that, contrary to the case of caustics by reflection
of planar curves, {\bf the set of base points of $\Phi_{|\hat Z}$ is never empty}. 
We give a general formula expressing the degree (with multiplicity) $\mdeg$ of $\Sigma_{ S}(\mathcal Z)$ in terms of intersection numbers of $\Z$
with a particular curve (called reflected polar curve) computed at the projection on $\Z$ of the base points of $\Phi_{|\hat Z}$.
As a consequence of our general result, we prove namely the following generic result (see Theorem \ref{generique} for precisions).
\begin{thm}\label{thmGENE}
Let $d\ge 1$. For a generic irreducible surface $\mathcal Z\subset\mathbb P^3$ of degree $d$ and for a generic $S\in\mathbb P^3$, we have
$$\mdeg \Sigma_S(\mathcal Z)=d(d-1)(8d-7) .$$
\end{thm}
We denote by $\mathcal H^\infty$ the plane at infinity of $\PP^3$:
$\mathcal H^\infty=\{[x:y:z:t]\in\PP^3\ :\ t=0\}$
and $\mathcal Z_\infty:=\mathcal Z\cap\mathcal H^\infty$.
In this study the {\bf umbilical curve} $\mathcal C_\infty$
plays a particular role.
Recall that $\mathcal C_\infty$ is the intersection of $\mathcal H^\infty$ with any sphere (see Section \ref{orthogonal}).

In practice, the degree of the caustic will namely depend on 
the position of $S$ with respect to the surface $\mathcal Z$,
to $\mathcal H^\infty$, to $\mathcal C_\infty$ and to 
the {\bf isotropic} tangent planes to $\mathcal Z$
(see Section \ref{orthogonal} for the notion of isotropic planes).

We illustrate this by a precise study of the degrees of caustics of a paraboloid $\mathcal Z$. In this case, $\mathcal Z_\infty$ is the union of two lines
intersecting at the focal point at infinity.
We will see that the caustic of the paraboloid is a surface if
the light position is outside $\mathcal Z_\infty$ and outside the focal points of $\mathcal Z$. 

Still in the case of the paraboloid, $\mathcal Z\cap\mathcal C_\infty$ is made of
two points $I$ and $J$ and the
tangent planes to $\mathcal Z$ at these two points are isotropic.
Moreover the revolution axis $\mathcal D$ of the paraboloid is the intersection
of these two tangent planes.
\begin{prop}\label{prop:paraboloid}
Let $\mathcal Z$ be the paraboloid $V(x^2+y^2-2zt)\subset\mathbb P^3$ of axis $\mathcal D=V(x,y)$ and let $S\in\mathbb P^3$.

If $S$ is a focal point of the paraboloid (either $F_1[0:0:1:0]$ or $F_2[0:0:1:2]$), then $\Sigma_S(\mathcal Z)$ is reduced to the
other focal point.

If $S\in\mathcal Z_\infty\setminus\{F_1\}$, then $\Sigma_S(\mathcal Z)$ is a planar curve of degree 2.

If $S\in \mathcal D$, then
the degree of $\mdeg\Sigma_S(\mathcal Z)=4$ if $S[0:0:0:1]$ and
$\mdeg\Sigma_S(\mathcal Z)=6$ elsewhere.

Assume now that $S$ is not on $\mathcal D\cup\mathcal Z_\infty\cup\{F_2\}$.

If $S$ is neither at infinity nor on the paraboloid $\mathcal Z$, then:
\begin{itemize}
\item 
If  $x_0^2+y_0^2=0$  and $z_0\ne t_0/2$ (i.e. if $S$ is on $\mathcal T_I\mathcal Z$
or on $\mathcal T_J\mathcal Z$ and on two other isotropic tangent planes to $\mathcal Z$), then $\mdeg\Sigma_S(\mathcal Z)=14$.
\item If $x_0^2+y_0^2=0$  and $z_0=t_0/2$ (i.e. if $S$ is on $\mathcal T_I\mathcal Z$ or on $\mathcal T_J\mathcal Z$ and on another isotropic tangent plane to $\mathcal Z$), then $\mdeg\Sigma_S(\mathcal Z)=12$.
\item 
If $x_0^2+y_0^2\ne 0$ and $x_0^2+y_0^2+(z_0- t_0/2)^2= 0$
(i.e. if $S$ is on three isotropic planes but neither on $\mathcal T_I\mathcal Z$ nor on  $\mathcal T_J\mathcal Z$), then $\mdeg \Sigma_S (\mathcal Z)=17$.
\item Otherwise (generic case: $S$ is on four isotropic tangent planes to $\mathcal Z$), then $\mdeg \Sigma_S (\mathcal Z)=18$.
\end{itemize}

If $S$ is at infinity, then:
\begin{itemize}
\item If $S\not\in \mathcal C_\infty$, then
 $\mdeg\Sigma_{S}(\mathcal Z)=12$.
\item 
If $S\in\mathcal C_\infty$, then 
$\mdeg\Sigma_S(\mathcal Z)=6$.
\end{itemize}

If $S$ is on $\mathcal Z$, then:
\begin{itemize}
\item 
If  $x_0^2+y_0^2+t_0^2=0$ (i.e. if the tangent plane to $\mathcal Z$ at $S$ is isotropic), then $\mdeg\Sigma_S(\mathcal Z)=12$.
\item 
If $x_0^2+y_0^2=0$ (i.e. if $S$ is on $\mathcal T_I\mathcal Z$
or on $\mathcal T_J\mathcal Z$), then $\mdeg\Sigma_S(\mathcal Z)=14$.
\item Otherwise $\mdeg\Sigma_S(\mathcal Z)=16$.
\end{itemize}

\end{prop}
The paper is organized as follows. 
Section \ref{orthogonal} is devoted to
the (complex) projectivization 
of orthogonality in the real euclidean affine 3-space (which plays a crucial role in the present work) and its link with the umbilical curve.
In Section \ref{reflected}, we 
construct the reflected lines. In Section \ref{ramification}, we use the
reflected lines and the ramification method to define the
caustic by reflection. In Section \ref{map}, we define the 
appropriate 2-covering $\hat Z$ of $\Z$ and the rational map $\Phi$.
In Section \ref{base}, we determine precisely the base points of  $\Phi_{|\hat Z}$. 
We define the reflected polar in section \ref{polar}
and use it in Section \ref{degree} to establish a formula for the degree of the caustic by reflection. In Section \ref{proofGENE},
we prove Theorem \ref{thmGENE}. 
In Section \ref{sec:paraboloid}, we prove Proposition \ref{prop:paraboloid}.
In Section \ref{bundle}, we precise a significative difference between the caustic by reflection studied in this paper and the focal loci of generic
varieties considered in \cite{Trifogli,CataneseTrifogli}.
In appendix \ref{decompose}, we study two families of caustics by reflection of surfaces
which are related to caustics by reflection of planar curves. 
\section{Affine and projective perpendicularity, link with umbilical conjugation}\label{orthogonal}
Consider the real euclidean affine 3-space 
${E}_{3}$ of
direction the 3-vector space $\mathbf E_{3}$ (endowed with some
fixed basis).
Let $\mathbf W:=(\mathbf E_3\oplus \mathbb R)\otimes\mathbb C$
(endowed with the induced basis).
Let $j:
{E}_{3}\hookrightarrow \mathbb P^3:=\mathbb{P}(\mathbf W)$ be the natural
map defined 
on coordinates by $j(x,y,z):=[x:y:z:1]$ for
every $\underline{m}(x,y,z)\in {E}_{3}$. 
We are interested ‫in
the interpretation
in the plane at infinity  of $\mathbb{P}^{3}$ of perpendicularity
at a point of two affine subvarieties of ${E}_{3}$.
Consider the two following quadratic forms
$$q(x,y,z)=x^2+y^2+z^2\ \mbox{on}\ \mathbf E_3\otimes \mathbb C \ \ \ \mbox{and}
\ \ \ Q(x,y,z,t)=x^2+y^2+z^2 \ \mbox{on}\ \mathbf W.$$
\begin{defi}
The \textbf{
umbilical curve} of $\mathbb P^3$ is the irreducible conic $\mathcal{C}_{\infty }:=V(Q_{|\mathcal H^\infty})\cong V(q) 
\subset\mathbb P(\mathbf E_3\otimes \mathbb C)$.
We call \textbf{cyclic point} any point of $\mathcal C_\infty$.
\end{defi}
We recall that every  (complex projectivized) sphere
contains $\mathcal{C}_{\infty }$.
It is worth noting that, for every $\underline{m}\in E_{3}$, we have the following classical diagram
$$
\begin{array}{ccccc}
 E_{3} & \overset{j}{\hookrightarrow } & \mathbb{P}(\mathbf{W}) & \overset{\Pi }{\longleftarrow } & \mathbf{W}\backslash \{0\}
\\ 
  & \overset{\xi_{\underline{m}}}{\searrow } &  &  &  \\ 
  &  & \mathcal{H}^{\infty } &  & 
\end{array}
$$
where $\Pi $ is the canonical projection and with 
$\xi_{\underline{m}}$ is
defined on coordinates by  
$\xi_{\underline{m}}(\underline{m}+(x,y,z))=[x:y:z:0]$. 
Given any vector subspace $\mathbf V\subset\mathbf E_3$,
the projective subspace  $\mathcal{V}:=\overline{j(\underline{m}+\mathbf V)}$ 
of $\mathbb P^3$ (where $\overline K$ denotes the Zariski closure of $K$) is the complex
projectivization of the affine subspace $V=\underline{m}+\mathbf V$ of $E_{3}$. We observe that  $\xi_{\underline{m}}(V)$ is
$\mathcal{V}_{\infty }:=\mathcal V\cap\mathcal H^\infty$.

An affine line ${L}$ (resp. an affine  plane ${H}$)
containing $\underline{m}\in {E}_{3}$ is
 defined by $\underline{m}+V_{1}$ (resp. $\underline{m}+V_{2}$) 
with $V_{i}$ an $i$-dimensional subspace of  $\mathbf E_{3}$.  Recall that the
(complex) projectivization  $\mathcal{L}$ of ${L}$ 
(resp. $\mathcal{H}$ of ${H)}$ is the projective line (resp. plane) of
$\mathbb P^3$ of equations obtained by homogeneization of the 
equations of $L$ (resp. $H$).

Hence, two lines ${L},L'$ containing $\underline{m}$ are
perpendicular at $\underline{m}$ if and only if their points at infinity are conjugated with respect the conic 
$\mathcal{C}_{\infty }$ .

A line ${L}$ and a plane ${H}$ containing $\underline{m}$ are
perpendicular if and only if $\mathcal H_{\infty }$ is the \textbf{polar} of $
\ell_{\infty }$ with respect to the conic $\mathcal{C}_{\infty }$ in
$\mathcal{H}^\infty\cong \mathbb{P}^{2}$.
This leads to the following definition of projective normal lines to a plane.
\begin{defi}
Let $\mathcal H=V(h)\subset\mathbb P^3$ (with $h\in\mathbf{W}^\vee\setminus\{\mathbf{0}\}$) be a projective plane and $m\in\mathcal H\setminus\mathcal H^\infty$.
The normal line $\mathcal N_m(\mathcal H)$ to $\mathcal H$ at $m$ is the line containing $m$
and $n_\infty(\mathcal H):=\Pi(\boldsymbol{\kappa}(\nabla h))$
with $\boldsymbol{\kappa}:\mathbf{W}\rightarrow\mathbf{W}$
defined on coordinates by $\boldsymbol{\kappa}(a,b,c,d):=
(a,b,c,0)$.
\end{defi}
\begin{rqe}
Given a projective plane $\mathcal H\subset\mathbb P^3$
($\mathcal H\ne\mathcal H^\infty$), if $n_{\infty}(\mathcal H)=[u:v:w:0]$ lies on the
umbilical (i.e. $(u,v,w)$ lies on the \textbf{isotropic cone} $V(q)$ in ${\mathbf E}_{3}\otimes \mathbb C$), 
then the line $\mathcal H_{\infty }$
is \textbf{tangent} to $\mathcal{C}_{\infty }$ at $n_\infty(\mathcal H)$
in $\mathcal{H}^\infty$. In this case we have $\mathcal N_m(\mathcal H)\subset \mathcal H$.
\end{rqe}
Let $m=\Pi(\mathbf{m})$ be a non singular point of $\mathcal Z\setminus \mathcal H^\infty$.
We write $\mathcal{T}_{m}(\mathcal{Z})$ for the {\bf projective
tangent plane} at $m$ to $\mathcal Z$. 
We also define the {\bf projective normal line} $\mathcal{N}_{m}(\mathcal{Z})$ at $m$ to $\mathcal Z$ is the projective normal
line to $\mathcal{T}_{m}(\mathcal{Z})$ at $m$, i.e. 
$\mathcal{N}_{m}(\mathcal{Z})$ is
the line containing $m$ and $n_{\infty,m}(\mathcal Z)
=\Pi(\boldsymbol{\kappa}(\nabla F(\mathbf{m})))$. 

Observe that the line at infinity  
$\mathcal T_{\infty,m }(\mathcal Z)$ of $\mathcal{T}_{m}(\mathcal{Z})$ is the \textbf{polar} of the point at infinity 
$n_{\infty,m }(\mathcal Z)$ of $\mathcal{N}_{m}(\mathcal{Z})$ with respect the conic 
$\mathcal{C}_{\infty }$.

Later, we will see that the base points of the reflected map can be seen
on the geometry on the normals at infinity with respect to  the umbilical. In particular isotropic tangent plane to $\mathcal Z$
containing $S$ will play some role.
\begin{defi}
A plane $\mathcal H=V(h)$ (with $h\in \mathbf{W}^\vee\setminus\{\mathbf{0}\}$) is said to be {\bf isotropic
} if $\nabla h$ is an isotropic vector for $Q$.
\end{defi}
\begin{rqe}
A plane $\mathcal H\subset
\mathbb P^3$ is isotropic if and only if either 
it is the plane at infinity $\mathcal H^\infty$ or if $n_\infty(\mathcal H)$
is in $\mathcal C_\infty$ (i.e. $\mathcal H$ contains its normal lines). 
\end{rqe}
In particular, the surface $\mathcal{Z}$ admits an isotropic tangent plane at one
of its nonsingular point $m[x:y:z:1]$ if and only if 
$m$ belongs to $V(Q(\nabla F),F)$. 
We note that the whole curve $\mathcal{C}_{\infty }$ is contained
in every complex projectivized sphere $
\mathcal{S}_{r}$ and that we  have $\mathcal{N}_{m}({\mathcal{S}_{r}})\subset \mathcal{
T}_{m}({\mathcal{S}_{r}})$ for all $m\in \mathcal{S}_{r}\setminus \mathcal{H}^\infty.$ This is also true for tori.

\bigskip Consider some particular points on $\mathcal{Z}$, playing a
particular role in the construction of the caustic map. Let $\mathcal 
B_0:=V(F,\Delta
_{\mathbf{S}}F,Q(\nabla F))$ in $\mathbb P^3$, the interpretation in the plane at infinity is the following one.
Let $m$ be a nonsingular point of 
$\mathcal{Z}\setminus \mathcal{H}^\infty$ then 
\begin{equation}\label{equiv}
m\in 
\mathcal{B}_0\Longleftrightarrow \left\{ 
\begin{array}{c}
S\in \mathcal{T}_{m}(\mathcal{Z}) \\ 
n_{\infty,m }(\mathcal Z)\in \mathcal{C}_{\infty }%
\end{array}%
\right. \Longleftrightarrow \left\{ 
\begin{array}{c}
(m S)\subset \mathcal T_{m}(\mathcal{Z}) \\ 
n_{\infty,m }(\mathcal Z)\in \mathcal{C}_{\infty }%
\end{array}%
\right. {\Longleftrightarrow }\left\{ 
\begin{array}{c}
(m S)_\infty\in \mathcal T_{\infty,m}(\mathcal Z)\\
\mathcal T_{\infty,m }(\mathcal Z)=\mathcal T_{n_{\infty,m}(\mathcal Z)}(\mathcal C_\infty)
\end{array}%
\right. .
\end{equation}
We observe that  $\mathcal{B}_0$ is in general a finite set, but that, for the unit sphere, $\mathcal{B}_0$ is a curve (the circle apparent contour of $\mathcal{Z}$ seen from $S$).

Let us now specify some additional notations used in this paper. 
We write $\mathbf{S}(x_0,y_0,z_0,t_0)\in\mathbf W\setminus\{0\}$.
For any $m[x:y:z:t]\in\mathbb P^3$, we will write $\mathbf{m}(x,y,z,t)\in \mathbf{W}\setminus\{0\}$.
For any $d'\ge 1$ and any $G \in Sym^{d'}(\mathbf{W}^\vee)$,
we write as usual $G_x,G_y,G_z,G_t\in Sym^{d'-1}(\mathbf{W}^\vee)$ for the partial derivatives of in $x$, $y$, $z$ and $t$ respectively.

\section{Reflected lines}\label{reflected}
The incident lines are the lines $(S\, m)$ with $m\in \mathcal Z$. 
We will define the reflected line $\mathcal R_m$ as the orthogonal symmetric of
$(S\, m)$ with respect to the tangent plane to $\Z$ at $m$.
To this end, we will define the orthogonal symmetric $\sigma(m)$ of $S$ with respect to the tangent plane to $\Z$ at $m$.
Let us first explain how one can give a sense to the notion of orthogonal symmetries in $\mathbb P^3$ by
complex projectivization of the euclidean affine situation.
\subsection{Orthogonal symmetric and map $\sigma$}
To every injective linear map $
\mathbf{W} \overset{f}{\rightarrow }  \mathbf{W}$,
corresponds a unique morphism
$\mathbb P(\mathbf{W}) \overset{\mathbb{P(}f)}\rightarrow   
\mathbb{P}(\mathbf{W})$.
Therefore, to every injective affine map 
$E_{3} \overset{g}{\rightarrow } E_{3}$, corresponds a unique algebraic map $
\mathbb P(\mathbf{W})  \overset{\iota(g)}\rightarrow   
\mathbb{P}(\mathbf{W})$.
This defines an injective groups homomorphism
$\iota:Aff(E_{3})\cong \mathbf{E}_3\rtimes Gl(\mathbf{E}_3)   
\rightarrow \mathbb{P}(Gl(\mathbf{W}))$
such that $
\iota(Is(E_{3}))=\iota(\mathbf{E}_3\rtimes O({\mathbf E}_ 3))
\subset \mathbb{P}(O(\hat Q))$, with $\hat Q=x^2+y^2+z^2+t^2$
on $\mathbf W$.
We apply this to the orthogonal symmetry $s_{H}$
with respect to some affine plane $H=V(\tilde h)\subseteq E_3$ with $\tilde h=ax+by+cz+d$. Recall that $s_H$ is 
defined by $s_{H}(P)=P-2\, \tilde h(P)\frac{\nabla \tilde h}{q(\nabla \tilde h))}$. This leads to the morphism $s_{\mathcal{H}}:=\iota (s_{h}):\mathbb P^3\rightarrow\mathbb P^3$
defined by $\mathbb P(\mathbf s_{h})$ with
$$\forall \mathbf P\in \mathbf{W},\ \ \ 
{\mathbf s}_{h}(\mathbf{P)}:=Q(\nabla h)\cdot\mathbf{P}
-2h(\mathbf{P})\cdot\boldsymbol{\kappa}(\nabla h)\in\mathbf{W},$$
with $\mathcal H=V(h)\subset\mathbb P^3$ and with $h=ax+by+cz+dt$ the homogeneized of $\tilde h$.
Now we extend this definition to any projective plane $\mathcal H\subset\mathbb P^3$ as follows.
\begin{defi}
Consider a plane $\mathcal H=V(h)\subseteq\mathbb P^3$
(with $h\in \mathbf{W}^\vee\setminus\{\mathbf{0}\}$).
We define the orthogonal symmetry $s_{\mathcal H}$ with respect
to $\mathcal H$ as the rational map given by $s_{\mathcal H}=\mathbb P
(\mathbf{s}_h)$ with
$$\forall\mathbf{P}\in \mathbf{W},\ \ \ 
\mathbf{s}_{h}(\mathbf P):=Q(\nabla h)\cdot \mathbf{P}-2
h(\mathbf{P})\cdot\boldsymbol{\kappa}(\nabla h)\in \mathbf{W}.$$
\end{defi}
We can notice that, when $\mathcal H\ne \mathcal H^\infty$,
$s_{\mathcal H}(P)$ is well defined in $\mathbb P^3$ except if $\mathcal H$ is an isotropic plane containing $P$
(see Proposition \ref{baseR}).  
For any non singular 
$m[x:y:z:t]\in\mathcal Z$, we define 
$\sigma(m):=s_{\mathcal T_m\mathcal Z }(S)={\mathbb P}(\boldsymbol{\sigma})(m)$ with
\begin{equation}\label{defisigma}
\boldsymbol{\sigma}:=Q(\nabla F)\cdot
  \mathbf{S}-2\Delta_{\mathbf{S}}F\cdot \boldsymbol{\kappa}
(\nabla F)\in\mathbf {W}
\end{equation}
on $\Pi^{-1}(\mathcal Z)$
with $\Delta_{\mathbf S}F$ the equation of the polar hypersurface
of $\mathcal P_{S}(\mathcal Z)$ given by
$\Delta_{\mathbf{S}}F:=DF\cdot \mathbf{S}$ (where $DF$ is the differential of $F$).
We extend the definition of $\boldsymbol{\sigma}(\mathbf{m})$ 
to any $\mathbf{m}\in\mathbf{W}\setminus\{0\}$.
Observe that $\boldsymbol{\sigma}$ defines a unique rational map $\sigma:\mathbb P^3\rightarrow\mathbb P^3$.
\begin{prop}\label{baseR}
The base points of the rational map $\sigma_{|\mathcal Z}$ are the singular points of 
$\mathcal Z$,
the points of tangency of $\Z$ with $\mathcal H^\infty$ and the points at which $\mathcal Z$ has an isotropic 
tangent plane containing $S$.
\end{prop}
\begin{proof}
We prove that the base points of $\sigma$ are the
points of $\mathbb P^3$ such that
$F_x=F_y=F_z=0 $
or such that
$Q(\nabla F)=0$ and $\Delta_{\mathbf{S}}F=0$.
It is easy to see that these points are base points of $\sigma$.
Now let $m=[x:y:z:t]$ be a point of $\mathbb P^3$ such that $\sigma(m)=0$.
\begin{itemize}
\item If $\Delta_{\mathbf{S}}(F)=0$, then, since ${\mathbf{S}}\ne 0$, we get that $Q(\nabla F)=0$.
\item If $Q(\nabla F)=0$, then either $\Delta_{\mathbf{S}}F=0$ or $\boldsymbol{\kappa}(\nabla F)=0$.
\item Assume now that $Q(\nabla F)\ne 0$. We have
$Q(\nabla F)\cdot {\mathbf{S}}=2\Delta_{\mathbf{S}}F\cdot
\boldsymbol{\kappa}(\nabla F)$.
This implies that $\boldsymbol{\kappa}(\nabla F)$ is non zero and proportional to ${\mathbf{S}}$ (which is also non zero), so that
$t_0=0$ and $0=y_0F_x-x_0F_y=z_0F_y-y_0F_z=x_0F_z-z_0F_x$.
Therefore, writing $\sigma^{(i)}$ for the $i$th coordinate
of $\boldsymbol{\sigma}$, we have
\begin{eqnarray*}
0&=& \sigma^{(1)}= Q(\nabla F)x_0-2(x_0F_x^2+y_0F_xF_y+z_0F_xF_z)\\
&=& Q(\nabla F)x_0-2(x_0F_x^2+x_0F_y^2+x_0F_z^2)=-Q(\nabla F)x_0.
\end{eqnarray*}
In the same way, we get
$0=\sigma^{(2)}= -Q(\nabla F)y_0$ and $0=\sigma^{(3)}= -Q(\nabla F)z_0$.
This contradicts the fact that $Q(\nabla F)\ne 0$ (since ${\mathbf{S}}\ne 0$).
\end{itemize}
\end{proof}
\begin{rqe}
Each $\sigma^{(i)}$ belongs to $Sym^{2(d-1)}(\mathbf{W}^\vee)$.
Moreover, for a general $(\mathcal Z, S)$, the set $V(F,F_x,F_y,F_z)$ is empty
and the base points of $\sigma_{|\mathcal Z}$ are the $2d(d-1)^2$ points of
$V(F,Q(\nabla F),\Delta_{\mathbf{S}}F)$.
\end{rqe}
\subsection{Reflected lines}
\begin{defi}
For any $m\in\mathcal Z$, the {\bf reflected line} $\mathcal R_m$ on $\mathcal Z$ at $m$ is the line $(m\sigma(m))$ when it is well defined.
\end{defi}
\begin{defi}\label{defiM}
We write $\mathcal M_{S,\mathcal Z}$ for the set of points 
$m\in\mathbb P^3$ such that $\mathbf{m}$ and $\boldsymbol{\sigma}(\mathbf{m})$
are proportional, i.e.
$\mathcal M_{S,\mathcal Z}:=\{m\in\mathbb P^3\ :\ \exists [\lambda_0:\lambda_1]
\in{\mathbb P}^1,\ \ \lambda_0\cdot \mathbf{m}+\lambda_1\cdot \boldsymbol{\sigma}(\mathbf{m})
=0\}$.
\end{defi}
Observe that $\mathcal R_m$ is well defined if $m\in\mathcal Z\setminus \mathcal M_{S,\mathcal Z}$.
\begin{prop}\label{B1}
We have $\mathcal Z\cap \mathcal M_{S,\mathcal Z}=\mathcal Z\cap(\base(\sigma)\cup\{S\}
       \cup\mathcal W)$,
with
$$\mathcal W:=\{m\in\mathcal Z\ :\ 
m=n_{\infty,m}(\mathcal Z),\ \ \Delta_{\mathbf{S}}F(m)\ne 0,\ Q(m)=0\},$$
with $n_{\infty,m}(\mathcal Z):=\Pi(\boldsymbol{\kappa}(\nabla F(\mathbf{m})))$.
\end{prop}
\begin{proof}
We prove $\mathcal Z \cap \mathcal M_{S,\mathcal Z}\subseteq\mathcal Z\cap(\base(\sigma)\cup\{S\}
       \cup\mathcal W)$, the inverse inclusion being clear.
Let $m\in (\mathcal Z \cap\mathcal M_{S,\mathcal Z})\setminus \base(\sigma)$. Observe that, due to the Euler identity, we have
$0=DF(\mathbf{m})\cdot \mathbf{m}$ and so
$0=DF \cdot \boldsymbol{\sigma}=-\Delta_{\mathbf{S}}F
      \cdot Q(\nabla F).$
If $\Delta_{\mathbf{S}}F=0$, then $\boldsymbol{\sigma}= Q(\nabla F)\cdot \mathbf{S}$,
so $m=\sigma(m)=S$.
If $ Q(\nabla F)=0$, then  
$\boldsymbol{\sigma}=-2\Delta_{\mathbf{S}}F\cdot\boldsymbol{\kappa}(\nabla F)$. So
$m=\sigma(m)=n_{\infty,m}(\mathcal Z)$; moreover $\Delta_{\mathbf{S}}F\ne 0$ and $Q=0$.
\end{proof}
\begin{lem}\label{dim3}
If $\dim \mathcal M_{S,\mathcal Z}=3$, then $\mathcal Z=\mathcal H^\infty$ or $V(\Delta_{\mathbf S}F,Q(\nabla F))=\mathbb P^3$.
\end{lem}
\begin{proof}
Due to Proposition \ref{B1}, we have $\mathcal Z \cap \mathcal M_{S,\mathcal Z}\subseteq\mathcal Z\cap(\base(\sigma)\cup\{S\}
       \cup\mathcal C_\infty)$.
Assume that $\dim \mathcal M_{S,\mathcal Z}=3$. This implies that $\base(\sigma)=\mathbb P^3$. So, due to the proof of Proposition
\ref{baseR}, we conclude that $\mathbb P^3=V(F_x,F_y,F_z)\cup V(\Delta_{\mathbf S}F,Q(\nabla F))$. So, either
$\mathbb P^3=V(F_x,F_y,F_z)$ (which implies $\mathcal Z=\mathcal H^\infty$) or $\mathbb P^3=V(\Delta_{\mathbf S}F,Q(\nabla F))$.
\end{proof}
\section{Caustic by reflection}\label{ramification}
Now, let us introduce some additional notations.
We define $N_{\mathbf{S}}(\mathbf{m})$ as the complexified homogenized square euclidean norm of ${\mathbf{S}}\mathbf{m}$ by
$$N_{\mathbf{S}}(\mathbf{m}):=(xt_0-x_0t)^2+(yt_0-y_0t)^2
      +(zt_0-z_0t)^2.$$
We will also consider the bilinear Hessian form $\hess_F$
of $F$ and its determinant $H_F$.
Let us see how to construct two maps $\psi=\psi^{\pm}:\mathcal Z\rightarrow\mathbb P^3$
such that the surface $\psi(\mathcal Z)$ is tangent to the reflected line $\mathcal R_m$ at $\psi(m)$, for a generic $m\in\mathcal Z$.
Observe first that $\psi(m)$ is in $\mathcal R_m$ implies that $\psi(m)$ can be rewritten
$$\boldsymbol{\psi}
   (\mathbf{m})=\lambda_0(\mathbf{m})\cdot \mathbf{m}+\lambda_1(\mathbf{m})\cdot \boldsymbol{\sigma}(\mathbf{m})
     \in \mathbf{W}\setminus\{0\} ,$$
with $[\lambda_0(\mathbf{m}):\lambda_1(\mathbf{m})]\in\mathbb P^1$ for every $m\in\mathcal Z$.
The main result of this section is the next theorem specifying the form of $\lambda_0$ and $\lambda_1$ (belonging to an integral extension of the ring $Sym(\mathbf{W}^\vee)$) 
which ensures that, for a generic $m\in\mathcal Z$, $\mathcal R_m$
is tangent to $\psi(\mathcal{Z})$ at $\psi(m)$.
\begin{thm}\label{THM}\label{thmsurface}
Let $\psi:U\rightarrow {\mathbb P}^3$ (with $U\subseteq\mathcal Z$) be given by
$$\boldsymbol{\psi}(\mathbf{m})=\lambda_0(\mathbf{m})\cdot \mathbf{m}+\lambda_1(\mathbf{m})\cdot 
\boldsymbol{\sigma}(\mathbf{m})\ \ \in\mathbf W ,$$
with $\lambda_0(\cdot)$ and $\lambda_1(\cdot)$ in
an integral extension of ${Sym(\mathbf{W}^\vee)}$ such that
\begin{equation}\label{formequadratique}
\alpha(\mathbf{m})(\lambda_0(\mathbf{m}))^2 +\beta(\mathbf{m})\lambda_0(\mathbf{m})
\lambda_1(\mathbf{m})+\gamma(\mathbf{m})(\lambda_1(\mathbf{m}))^2=0
\end{equation}
with $\alpha,\beta,\gamma\in Sym(\mathbf{W}^\vee)$ given by
\begin{equation}\label{alpha0}
\alpha:=\Delta_{\mathbf{S}}F\ \ \in Sym^{d-1}(\mathbf{W}^\vee),\end{equation}
\begin{equation}\label{beta0}
\beta:=-2 \left[
  \hess F(\mathbf{S}, \boldsymbol{\sigma})
+(\Delta_{\mathbf{S}}F)^2(F_{xx}+F_{yy}+F_{zz})\right]\ \ \in 
Sym^{3d-4}(\mathbf{W}^\vee)
\end{equation}
and
\begin{equation}\label{gamma0}\gamma:=- \frac{4\Delta_{\mathbf{S}}F}{(d-1)^2}N_{\mathbf{S}}\, H_F \ \ \in 
Sym^{5d-7}(\mathbf{W}^\vee).\end{equation}
Then, for every 
$m\in \mathcal Z\setminus V(t Q(\nabla F))$,  
the reflected line $\mathcal R_m$ is tangent to $\psi(\mathcal Z)$ at $\psi(m)$.
\end{thm}
It will be useful to introduce
$$\forall (\mathbf{m},\lambda_0,\lambda_1)\in\mathbf{W}\times\mathbb C^2,\ \ Q_{{\mathbf{S}},F}(\mathbf{m},\lambda_0,\lambda_1)=\alpha(\mathbf{m})\lambda_0^2+\beta(\mathbf{m})\lambda_0\lambda_1+\gamma(\mathbf{m})
      \lambda_1^2.$$
One may notice that, for a  fixed $\mathbf{m}$, $Q_{{\mathbf{S}},F}(\mathbf{m},\lambda_0,\lambda_1)$ is a quadratic form in 
$(\lambda_0,\lambda_1)$.
Roughly speaking, Theorem \ref{THM}
states that the image of $\mathcal Z$ by $\psi(\cdot)=\lambda_0(\cdot)\cdot \Idd+\lambda_1(\cdot)\cdot\sigma(\cdot)$
(for some $\lambda_0,\lambda_1\in Sym(\mathbf{W}^\vee)[\sqrt{\beta^2-4\alpha\gamma}]$)
corresponds to a part of the envelope of the reflected lines $\mathcal R_m$. More precisely:
\begin{defi}\label{deficaustique}
The {\bf caustic by reflection} $\Sigma_{ S}(\mathcal Z)$ 
of $\mathcal Z$ from $S$ is the Zariski closure
 of the following set
$$\{P\in{\mathbb P^3}\ :\ \exists m\in\Z,\ \exists [\lambda_0:\lambda_1]\in\mathbb P^1,\ \ 
   Q_{{\mathbf{S}},F}(\mathbf{m},\lambda_0,\lambda_1)=0\ \mbox{and}\ 
  {\mathbf P}=\lambda_0 \cdot \mathbf{m}+\lambda_1\cdot\boldsymbol{\sigma}(\mathbf{m})\}.$$
\end{defi}
\begin{rqe}
If $\mathcal Z\subseteq V(\Delta_{\mathbf{S}}F,
 (F_x^2+F_y^2+F_z^2)\hess F(\mathbf{S},\mathbf{S}))$, then
(\ref{formequadratique}) becomes $0=0$ on $\mathcal Z$ and $\sigma_{S,\mathcal Z}(\mathcal Z)$
is either $\{S\}$ or empty. If it is $S$ (i.e. if 
$\Delta_{\mathbf S}F=0$ in $\mathbb C[x,y,z,t]$
and if $\mathcal Z\not\subseteq
V(F_x^2+F_y^2+F_z^2)$), we set $\Sigma_S(\mathcal Z)=\{S\}$.
\end{rqe}
Theorem \ref{THM} states that the points of the caustic
$\Sigma_{S}(\mathcal Z)$ corresponding to $m\in\mathcal Z$
are the points of coordinates $\boldsymbol{\psi^\pm}(\mathbf m)$
with
\begin{equation}\label{solutions}
\boldsymbol{\psi^\pm}(\mathbf m)= \left(\tilde\beta(\mathbf m)\pm\sqrt{\vartheta
(\mathbf m)}\right)
\cdot \mathbf{m}+
\Delta_{\mathbf S}F(\mathbf m)\cdot
  \boldsymbol{\sigma}(\mathbf m)\in\mathbb C^4,
\end{equation}
with $\tilde\beta:=-\beta/2$
and $\vartheta:=\tilde\beta-
\alpha\gamma$.
Let us observe that if $\vartheta$ is a square in
$\mathbb C[x,y,z,t]/(F)$, then, on $\mathcal Z$, (\ref{solutions}) corresponds to two rational 
maps $\psi^\pm:\mathbb P^3\rightarrow\mathbb P^3$ 
and the caustic by reflection $\Sigma_S(\mathcal Z)$ is
the union of the Zariski closures of $\psi^+(\mathcal Z)$ and of
$\psi^-(\mathcal Z)$.
Let us give some examples.
\begin{exa}[A singular caustic of the saddle surface]
Let us study the caustic by reflection of $\mathcal Z=V(xy-zt)$ from
$S=[0:0:1:0]$. Observe that $\alpha=-t$, 
$\beta=0$ and $\gamma=4t^3$.
So (\ref{formequadratique}) becomes $\lambda_0^2-4t^2\lambda_1^2=0$ 
(if $t\ne 0$). Hence $\Sigma_{S}(\mathcal Z)$ is the union of the Zariski closure
of the images of $\mathcal Z$ by the two rational maps 
$\psi^\pm:\mathbb P^3\mapsto\mathbb P^3$ defined on coordinates by
$\boldsymbol{\psi^\pm}(x,y,z,t)=(2ty\pm 2tx,2tx\pm 2ty,
x^2+y^2-t^2\pm 2tz,\pm 2t^2)$.
Noting that on $\mathcal Z$, $tz=xy$, we obtain that the Zariski
closure of $\psi^\pm(\mathcal Z)$ is the parabola $V(X\mp Y,\pm TZ-(X^2-T^2)/2)$
and so that the caustic $\Sigma_S(\mathcal Z)$ is the union of these two curves
(which are parabolas contained in two orthogonal planes).
\end{exa}
\begin{exa}[The double-butterfly caustic of the saddle surface]\label{butterfly}
We are interested in the caustic by reflection of $\mathcal Z=V(xy-zt)$ from
$S=[0:0:1:1]$. We have $\alpha=-z-t$, 
$\tilde\beta=-2(x^2+y^2-zt)$ and 
$\gamma=4(z+t)(x^2+y^2+z^2+t^2-2tz)$
and so
$\vartheta=4(x^4+y^4+z^4+t^4-z^2t^2+2x^2y^2+x^2z^2+y^2z^2+x^2t^2+y^2t^2)$.
Since $\vartheta(x,y,xy,1)$ is not a square in
$\mathbb C[x,y]$, we conclude that $\vartheta$ is not
a square in $\mathbb C[x,y,z,t]/(F)$. In this case, the coordinates of
$\boldsymbol{\psi^\pm}$ are in an extension of $\mathbb C[x,y,z,t]$
and do not corresponds to rational maps on $\mathbb P^3$
(see Figure \ref{fig:papillon} for a representation of this caustic).
\end{exa}

\begin{figure}[!ht]
\centering
\makebox{
\includegraphics[scale=.8]{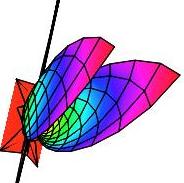} \includegraphics[scale=.6]{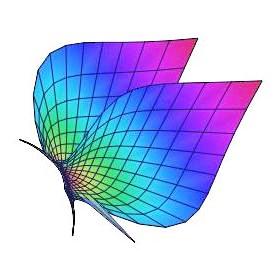}
\includegraphics[scale=0.7]{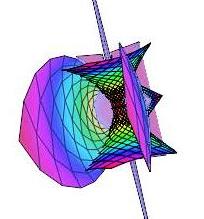}}
\caption{\label{fig:papillon} {The picture on the right represents the 
caustic by reflection of $V(zt-xy)$ from $[0:0:1:0]$ (corresponding
to the points of $V(zt-xy)$ in the chart $t=1$ for $x,y\in[-4,4]$).
This caustic is obtained by gathering the two sheets 
given by (\ref{solutions}) and represented in the two first pictures.}}
\end{figure}

\begin{exa}\label{exaparaboloid}
Let $\mathcal Z$ be the paraboloid $V((x^2+y^2-2zt)/2)\subset\mathbb P^3$.

The caustic by reflection $\mathcal Z$
from its focal point $F_1=[0:0:1:0]$ 
is its other focal point $F_2=[0:0:1:2]$.
This can be quickly shown with our Theorem \ref{thmsurface}
(since ${\alpha}=-t$, ${\beta}=-4t^2$ and
${\gamma}=4t^3$ and so (\ref{formequadratique}) admits a unique
solution $[\lambda_0:\lambda_1]\in\mathbb P^1$ which is $[-2t:1]$.
The unique ramification point $M_m$ associated to $m$ is then
$M_m=[0:0:-2zt+x^2+y^2-t^2:-2t^2]=F_2$ (since $x^2+y^2=2zt$).

If light position $S$ is another point of $\mathcal H^\infty$, then
$$\vartheta=(x_0^2+y_0^2)[(x^2+y^2+t^2)^2(x_0^2+y_0^2)-4(x^2+y^2+t^2)(xx_0+yy_0-
z_0t)z_0t-4(xx_0+yy_0-z_0t)^2t^2]$$
is not a square in $\mathbb C[x,y,z,t]/(F)$
unless if $x_0^2+y_0^2+z_0^2=0$ or $x_0^2+y_0^2=0$ (see Proposition \ref{cas1} for the case when
$x_0^2+y_0^2=0$ and $t_0=0$).\footnote{To prove that $\delta$
is not a square in $\mathbb C[x,y,z,t]/(F)$, it is enough to see 
that there exists no polynomial $P\in\mathbb C[x,y]$ such that $(P(x,y))^2=\delta(x,y,(x^2+y^2)/2,1)$.} The fact that $\vartheta$ is not a square
means that the caustic map $\Phi_{S,F}$ cannot be decomposed
in two rational maps on $\mathbb P^3$.
\end{exa}

The end of this section is devoted to the proof of Theorem 
\ref{THM}.
\begin{proof}[Proof of Theorem \ref{THM}]
Let $m[x:y:z:t]\in\mathcal Z\setminus V(t(Q(\nabla F)))$. 
We will use several times the Euler identity 
($xG_x+yG_y+zG_y+tG_t=d_1G$ if $G$ is in $Sym^{d_1}(W^\vee)$). 
We use the idea of ramification (used for example in \cite{Trifogli,CataneseTrifogli}). 
The points of the caustic corresponding to $m$ are the points $\Pi(\lambda_0 \cdot\mathbf{m}+\lambda_1\cdot \boldsymbol{\sigma}(\mathbf{m}))$
with $[\lambda_0:\lambda_1]\in\mathbb P^1$ such that the rank of the 
Jacobian matrix $J$ of 
$$j:(\mathbf{m},\lambda_0,\lambda_1)\mapsto (\lambda_0 \cdot\mathbf{m}+\lambda_1 \cdot\boldsymbol{\sigma}(\mathbf{m}),F(\mathbf{m}))$$
is less than 5. We have
$$J:=\left(\begin{array}{cccccc}\lambda_0+\lambda_1 \sigma^{(1)}_x&\lambda_1 \sigma^{(1)}_y&\lambda_1 \sigma^{(1)}_z
     &\lambda_1 \sigma^{(1)}_t&x&\sigma^{(1)}\\
    \lambda_1 \sigma^{(2)}_x&\lambda_0+\lambda_1 \sigma^{(2)}_y&\lambda_1 \sigma^{(2)}_z&\lambda_1 \sigma^{(2)}_t&y&\sigma^{(2)}\\
    \lambda_1 \sigma^{(3)}_x&\lambda_1 \sigma^{(3)}_y&\lambda_0+\lambda_1 \sigma^{(3)}_z&\lambda_1 \sigma^{(3)}_t&z&\sigma^{(3)}\\
    \lambda_1 \sigma^{(4)}_x&\lambda_1 \sigma^{(4)}_y&\lambda_1 \sigma^{(4)}_z&\lambda_0+\lambda_1 \sigma^{(4)}_t&t&\sigma^{(4)}\\
    F_x&F_y&F_z&F_t&0&0\end{array}\right),$$
with $\sigma^{(i)}$ the $i$th coordinates of $\boldsymbol{\sigma}$.
\begin{enumerate}
\item Let us explain this
briefly. Let $\psi(\cdot)$ of the form
$\boldsymbol{\psi}(\mathbf{m}')=\lambda_0(\mathbf{m}')\cdot\mathbf{m}'+\lambda_1(\mathbf{m}')\cdot\boldsymbol{\sigma}(\mathbf{m}').$
We define the following property
\begin{equation}\label{(*)}
\mbox{the line }(m\sigma(m))\mbox{ is tangent to }\psi(\mathcal Z)\mbox{ at }\psi(m).
\end{equation}
Recall that we have assumed $(Q(\nabla F))(\mathbf{m})\ne 0$.
Assume for example $F_x(\mathbf{m})\ne 0$ (the proof is similar if we replace
$F_x$ by $F_y$ or by $F_z$).
Now, Property (\ref{(*)})
means that there exists $A\in\mathbf{W}^\vee\setminus\{0\}$ such that
$$ A(\mathbf{m})=0,\   A(\boldsymbol{\sigma}(\mathbf{m}))=0,\ \ 
     A((D\boldsymbol{\psi}(\mathbf{m} )\cdot\left(\begin{array}{c}F_y(\mathbf{m} )\\-F_x(\mathbf{m} )\\0\\0\end{array}\right))=0,$$
$$    A((D\boldsymbol{\psi}(\mathbf{m} ) )\cdot\left(\begin{array}{c}F_z(\mathbf{m} )\\0\\-F_x(\mathbf{m} )\\0\end{array}\right))=0\ \ \mbox{and}\ \ 
     A((D\boldsymbol{\psi}(\mathbf{m} ) )\cdot\left(\begin{array}{c}F_t(\mathbf{m} )\\0\\0\\-F_x(\mathbf{m} )\end{array}\right))=0,$$
and so that
$$ A(\mathbf{m})=0,\   A(\boldsymbol{\sigma}(\mathbf{m}))=0,\ \ 
     A(\boldsymbol{\psi}_x(\mathbf{m})) F_y(\mathbf{m})=F_x(\mathbf{m}) A(\boldsymbol{\psi}_y(\mathbf{m})),$$
$$     A(\boldsymbol{\psi}_x(\mathbf{m})) F_z(\mathbf{m})=F_x(\mathbf{m})A(\boldsymbol{\psi}_z(\mathbf{m}))\ \ \mbox{and}\ \ 
    A(\boldsymbol{\psi}_x(\mathbf{m})) F_t(\mathbf{m})=F_x(\mathbf{m})A(\boldsymbol{\psi}_t(\mathbf{m})).$$
Therefore, by taking $b:= A(\boldsymbol{\psi}_x(\mathbf{m}))/F_x(\mathbf{m})$, 
$$ A(\mathbf{m})=0,\ \  A(\boldsymbol{\psi}_x(\mathbf{m}))=F_x(\mathbf{m})b,\ \ A(\boldsymbol{\psi}_y(\mathbf{m}))=F_y(\mathbf{m})b,$$
$$ A(\boldsymbol{\psi}_z(\mathbf{m}))=F_z(\mathbf{m})b\ \mbox{ and }
 A(\boldsymbol{\psi}_t(\mathbf{m}))=F_t(\mathbf{m})b  $$
and so that the rank of the following matrix is strictly less than 5
$$\hat J:=\left(\begin{array}{cccccc}\boldsymbol{\psi}_x(\mathbf{m})&\boldsymbol{\psi}_y(\mathbf{m})&\boldsymbol{\psi}_z(\mathbf{m})&\boldsymbol{\psi}_t(\mathbf{m})& \mathbf{m}& \boldsymbol{\sigma}(\mathbf{m})\\
         F_x&F_y&F_z&F_t&0&0\end{array}\right)\in Mat_{5,6}(\mathbb C).$$
Let us write $C_i$ the $i$-th column of $J$.
We observe that the four first columns of $\hat J$ are respectively equal to
$C_1+(\lambda_1)_xC_6+(\lambda_0)_xC_5$, $C_2+(\lambda_1)_yC_6+(\lambda_0)_yC_5$,
$C_3+(\lambda_1)_zC_6+(\lambda_0)_zC_5$ and $C_4+(\lambda_1)_tC_6+(\lambda_0)_tC_5$.
Therefore the $J$ and $\hat J $ have the same rank and so (\ref{(*)}) means that $rank(J)<5$.
\item 
Now we observe that, on $\mathcal Z$, $xC_1+yC_2+zC_3+tC_4=\lambda_0C_5+\lambda_1C_6$.
Since $t\ne 0$, $C_4$ is a linear combination of the other columns and so the rank of $J$ is strictly
less than 5 if and only if the following determinant is null:
$$ D(\mathbf{m},\lambda_0,\lambda_1):=\left|\begin{array}{ccccc}\lambda_0+\lambda_1\sigma^{(1)}_x&\lambda_1 \sigma^{(1)}_y&\lambda_1 \sigma^{(1)}_z& x&\sigma^{(1)}\\
    \lambda_1\sigma^{(2)}_x&\lambda_0+\lambda_1\sigma^{(2)}_y&\lambda_1\sigma^{(2)}_z&y&\sigma^{(2)}\\
    \lambda_1\sigma^{(3)}_x&\lambda_1\sigma^{(3)}_y&\lambda_0+\lambda_1 \sigma^{(3)}_z&z&\sigma^{(3)}\\
    \lambda_1 \sigma^{(4)}_x&\lambda_1\sigma^{(4)}_y&\lambda_1 \sigma^{(4)}_z&t&\sigma^{(4)}\\
    F_x&F_y&F_z&0&0\end{array}\right|. $$
Now let us define
$$\boldsymbol{\tau}:=
Q(\nabla F)\cdot{\mathbf{S}}+2\frac{(xt_0-x_0t)F_x+(yt_0-y_0t)F_y+(zt_0-z_0t)F_z}t\cdot\boldsymbol{\kappa}(\nabla F).$$
Observe that 
$\boldsymbol{\tau}=\boldsymbol{\sigma}+\frac{2t_0dF}t
\boldsymbol{\kappa}(\nabla F)$ (due to the Euler identity).
Therefore, on $\mathcal Z$, we have $\boldsymbol{\sigma}=\boldsymbol{\tau}$.
Now we observe that, on $\mathcal Z$, we have
\begin{equation}\label{formuleD}
 D(\mathbf{m},\lambda_0,\lambda_1)=\left|\begin{array}{ccccc}\lambda_0+\lambda_1\tau^{(1)}_x&\lambda_1 \tau^{(1)}_y&\lambda_1 \tau^{(1)}_z& x&\tau^{(1)}\\
    \lambda_1\tau^{(2)}_x&\lambda_0+\lambda_1\tau^{(2)}_y&\lambda_1\tau^{(2)}_z&y&\tau^{(2)}\\
    \lambda_1\tau^{(3)}_x&\lambda_1\tau^{(3)}_y&\lambda_0+\lambda_1 \tau^{(3)}_z&z&\tau^{(3)}\\
    \lambda_1\tau^{(4)}_x&\lambda_1\tau^{(4)}_y&\lambda_1 \tau^{(4)}_z&t&\tau^{(4)}\\
    F_x&F_y&F_z&0&0\end{array}\right|,
\end{equation}
with $\tau^{(i)}$ the $i$th coordinate of $\boldsymbol{\tau}$.
Indeed, if we write $L_i$ the $i$-th line of the matrix (with $\boldsymbol{\sigma}$) used in the definition of $D$ and if we write
$\tilde L_i$ the $i$-th line of the matrix (with $\boldsymbol{\tau}$) appearing in the above formula, we obtain (due to the Euler identity) that, on $\mathcal Z$, 
we have $\tilde L_4=L_4$, $\tilde L_5=L_5$ and
$\tilde L_1=L_1+ \lambda_1\frac {2t_0 d} t F_x L_5$, $\tilde L_2=L_2+ \lambda_1\frac {2t_0 d} t F_y L_5$,
$\tilde L_3=L_3+ \lambda_1\frac {2t_0 d} t F_z L_5$. 
\item 
On $\mathcal Z$, we have
\begin{equation}\label{polynomeD}
 D(\mathbf{m},\lambda_0,\lambda_1)=\alpha_1(\mathbf{m})\lambda_0^2+\beta_1(\mathbf{m})\lambda_0\lambda_1+\gamma_1(\mathbf{m})\lambda_1^2,
\end{equation}
where $\alpha_1$, $\beta_1$ and $\gamma_1$ can be expressed as follows 
(due to Euler's identity ensuring that
$-xF_xt_0-yF_yt_0-zF_zt_0+tx_0F_x+ty_0F_y+tz_0F_z=t\Delta_{\mathbf{S}}F$ on $\mathcal Z$)
\begin{equation}\label{alpha}
\alpha_1:= Q(\nabla F)t\Delta_{\mathbf{S}}F=tQ(\nabla F)\alpha 
\end{equation}
\begin{equation}
\beta_1:=-\frac 2 tQ(\nabla F)B
\end{equation}
\begin{equation}
\gamma_1:=-4t^{-1}N_{\mathbf{S}}.Q(\nabla F).\Delta_{\mathbf{S}}F.h_F
\end{equation}
with the following definitions of $h_F$ and $B$.
First, on $\mathcal Z$, we have
\[
h_F:=\left|\begin{array}{cccc}F_{xx}&F_{xy}&F_{xz}&F_x\\F_{xy}&F_{yy}&F_{yz}&F_y\\F_{xz}&F_{yz}&F_{zz}&F_z\\F_x&F_y&F_z&0\end{array}\right|=\frac{t^2}{(d-1)^2}H_F,
\]
where $H_F$ is the Hessian determinant of $F$
\footnote{Indeed, if we write $\hat C_i$ for the $i$-th column of $\hess F$, due to the Euler formula, on $\mathcal Z$, we
have $\hat C_4=\frac{d-1}t\nabla F-(x\hat C_1+y\hat C_2+z\hat C_3)/t$ (where $\nabla F$ is the gradient of $F$); 
therefore
$$H_F:=\left|\begin{array}{cccc}F_{xx}&F_{xy}&F_{xz}&F_{xt}\\F_{xy}&F_{yy}&F_{yz}&F_{yt}\\F_{xz}&F_{yz}&F_{zz}&F_{zt}\\F_{xt}&F_{yt}&F_{zt}&F_{tt}    \end{array}\right|=\frac{d-1}t \left|\begin{array}{cccc}F_{xx}&F_{xy}&F_{xz}&F_x\\F_{xy}&F_{yy}&F_{yz}&F_y\\F_{xz}&F_{yz}&F_{zz}&F_z\\F_{xt}&F_{yt}&F_{zt}&F_t

    \end{array}\right| .$$
Now, if we write $\hat L_i$ the $i$-th line of the above matrix, using again the Euler identity, on $\mathcal Z$, we have
$\hat L_4=\frac {d-1}t(F_x\ \ F_y\ \ F_z\ \ 0)-(x\hat L_1+y\hat L_2+z\hat L_3)/t$
and we get $H_F=(d-1)^2h_F/t^2$.}. Therefore
\begin{equation}\label{gamma}
\gamma_1 =  -\frac{4t}{(d-1)^2}N_{\mathbf{S}}.Q(\nabla F).\Delta_{\mathbf{S}}F.H_F=tQ(\nabla F)\gamma.
\end{equation}
Second $B:=\delta_{x}F_{xx}+\delta_{y}F_{yy}+\delta_{z}F_{zz}+2(\varepsilon_{x,y}F_{xy}+\varepsilon_{x,z}F_{xz}+\varepsilon_{y,z}F_{yz}),$
with
\begin{eqnarray*}
\delta_x&:=&(x_0t-xt_0)^2(F_y^2+F_z^2)+((t_0y-ty_0)F_y+(t_0z-tz_0)F_z)^2\\
&=&(x_0t-xt_0)^2(F_y^2+F_z^2)+(t_0(yF_y+zF_z)-t(y_0F_y+z_0F_z))^2\\
&=&(x_0t-xt_0)^2(F_y^2+F_z^2)+(t_0(xF_x+tF_t)+t(y_0F_y+z_0F_z))^2\\
&=&(x_0t-xt_0)^2(F_y^2+F_z^2)+((t_0x-x_0t)F_x+t\Delta_{\mathbf{S}}F)^2\\
&=&x^2t_0^2(F_x^2+F_y^2+F_z^2)+2xt_0t[-x_0(F_x^2+F_y^2+F_z^2)+F_x\Delta_{\mathbf{S}}F]+\\
&\ &+x_0^2t^2(F_x^2+F_y^2+F_z^2)+t^2(\Delta_{\mathbf{S}}F)^2-2x_0t^2F_x\Delta_{\mathbf{S}}F\\
&=&x^2t_0^2(F_x^2+F_y^2+F_z^2)+t\Delta_{\mathbf{S}}F(2xt_0F_x-2x_0tF_x)+\\
&\ &+t(F_x^2+F_y^2+F_z^2)(x_0^2t-2xx_0t_0)+t^2(\Delta_{\mathbf{S}}F)^2,
\end{eqnarray*}
$\delta_y$ (resp. $\delta_z$) being obtained from $\delta_x$ by interverting $x$ and $y$
(resp. $x$ and $z$) and
\begin{eqnarray*}
\varepsilon_{x,y}&:=& -(t_0x-x_0t)F_y((t_0x-x_0t)F_x+(t_0z-z_0t)F_z)\\
&\ & -(t_0y-y_0t)F_x((t_0y-y_0t)F_y+(t_0z-z_0t)F_z)+(t_0x-x_0t)(t_0y-y_0t) F_z^2\\
&=&(t_0x-x_0t)F_y[t_0yF_y+t_0tF_t+x_0tF_x+z_0tF_z]+\\
&\ &+(t_0y-y_0t)F_x[t_0xF_x+t_0tF_t+y_0tF_y+z_0tF_z]+\\
&\ &+(t_0x-x_0t)(t_0y-y_0t) F_z^2\\
&=&(t_0x-x_0t)F_y[t_0yF_y+t\Delta_{\mathbf{S}}F-ty_0F_y]+(t_0y-y_0t)F_x[t_0xF_x+t\Delta_{\mathbf{S}}F-tx_0F_x]+\\
&\ &+(t_0x-x_0t)(t_0y-y_0t) F_z^2\\
&=&t_0^2xy(F_x^2+F_y^2+F_z^2)+t\Delta_{\mathbf{S}}F((t_0x-x_0t)F_y+(t_0y-y_0t)F_x)\\
&\ &+t(F_x^2+F_y^2+F_z^2)(tx_0y_0-t_0(y_0x+yx_0))
\end{eqnarray*}
$\varepsilon_{x,z}$ (resp. $\varepsilon_{y,z}$) being obtained from $\varepsilon_{x,y}$ by interverting
$y$ and $z$ (resp. $x$ and $z$).
On $\mathcal Z$, we have
$$0=xF_x+yF_y+zF_z+tF_t\ \mbox{and}\ (d-1)F_w=xF_{xw}+yF_{yw}+zF_{zw}+tF_{tw}, \ \forall w\in\{x,y,z,t\}.$$
Therefore
$$0=x^2F_{xx}+y^2F_{yy}+z^2F_{zz}+t^2F_{tt}+2(xyF_{xy}+xzF_{xz}+xtF_{xt}+yzF_{yz}+ytF_{yt}+ztF_{zt})$$
and so
$$B=(F_x^2+F_y^2+F_z^2)(b_1+b_2+b_3)+2t\Delta_{\mathbf{S}}Fb_4+t^2(\Delta_{\mathbf{S}}F)^2(F_{xx}+F_{yy}+F_{zz}),$$
with
$$b_1=-t_0^2(t^2 F_{tt} +2t(xF_{xt}+yF_{yt}+zF_{zt}))=
     -t_0^2t(2(d-1)F_t-tF_{tt}),$$
$$b_2=t^2(x_0^2F_{xx}+y_0^2F_{yy}+z_0^2F_{zz}+2x_0y_0F_{xy}+2x_0z_0F_{xz}+2y_0z_0F_{yz}),$$
\begin{eqnarray*}
b_3&=&-2tt_0\sum_{w\in\{x,y,z\}}(w_0(xF_{xw}+yF_{yw}+zF_{zw}))\\
&=&2tt_0\sum_{w\in\{x,y,z\}}(w_0(tF_{tw}-(d-1)F_w)),
\end{eqnarray*}
\begin{eqnarray*}
b_4&=& \sum_{w\in\{x,y,z\}}F_w((t_0x-tx_0)F_{xw}+(t_0y-ty_0)F_{yw}+(t_0z-tz_0)F_{zw})\\
&=& \sum_{w\in\{x,y,z\}}F_w(t_0(d-1)F_w-t(x_0F_{xw}+y_0F_{yw}+z_0F_{zw}+t_0F_{wt})).
\end{eqnarray*}
Putting all these terms together, we get that $B$ is equal to
$$t^2\left[ Q(\nabla F)  \cdot\hess_F(\mathbf{S},\mathbf{S})
-2\Delta_{\mathbf{S}}F \cdot\hess_F(\mathbf{S},\boldsymbol{\kappa}(\nabla F))
+(\Delta_{\mathbf{S}}F)^2(F_{xx}+F_{yy}+F_{zz})\right].
$$
and so
$B=t^2\left[\hess_F({\mathbf{S}},\boldsymbol{\sigma})
+(\Delta_{\mathbf{S}}F)^2(F_{xx}+F_{yy}+F_{zz})\right],$
which leads to
\begin{equation}\label{beta}
\beta_1=tQ(\nabla F)\beta.
\end{equation}
Hence the points of the caustic associated to $m$ are the points $\Pi(\lambda_0 \cdot\mathbf{m}+\lambda_1\cdot \boldsymbol{\sigma}(\mathbf{m}))$
where $[\lambda_0:\lambda_1]\in\mathbb P^1$ satisfies
\begin{equation}\label{EQ1}
\alpha_1(\mathbf{m})\lambda_0^2+\beta_1(\mathbf{m})\lambda_0\lambda_1+\gamma_1(\mathbf{m})\lambda_1^2=0,
\end{equation}
with $\alpha_1$, $\beta_1$ and $\gamma_1$ given by (\ref{alpha}), (\ref{beta}) and (\ref{gamma}).
Now, since $tQ(\nabla F)\ne 0$,
(\ref{EQ1}) means that $\alpha(\mathbf{m})\lambda_0^2+\beta(\mathbf{m})\lambda_0\lambda_1+\gamma(\mathbf{m})\lambda_1^2=0$.
\end{enumerate}
\end{proof}
\section{Covering space $\hat Z$ and rational map $\Phi$}\label{map}
We consider the algebraic covering space $\hat Z$ of $\mathcal Z$ given by
$$\hat Z:=\{(m,[\lambda_0:\lambda_1])\in\mathcal Z\times\mathbb P^1\ :\ 
   Q_{{\mathbf{S}},F}(\mathbf{m},\lambda_0,\lambda_1)=0\}.$$
This set is a subvariety of a particular algebraic variety 
denoted $\mathbb F_{(3)}(-2d+3,0)$ (by extending the notations used by Reid in \cite[Chapter 2]{Reid})
which corresponds to the cartesian product of sets $\mathbb P^3\times \mathbb P^1$ endowed with 
an unusual structure of algebraic variety based on the following definition of multidegree
$\mbox{multideg}(P)$ for $P\in
Sym(\mathbf{W}^\vee)[\lambda_0,\lambda_1]\cong\mathbb C[x,y,z,t,\lambda_0,\lambda_1]$:
$$\mbox{multideg}(x^{a'}y^{b'}z^{c'}t^{d'}\lambda_0^{e'}\lambda_1^{f'})=(a'+b'+c'+d'+(2d-3)e',e'+f'). $$
With this notion of multidegree, we have
$Sym(\mathbf{W}^\vee)[\lambda_0,\lambda_1]=\bigoplus_{k,\ell\ge 0}C_{k,\ell},$
where $C_{k,\ell}$ denotes the homogeneous component of multidegree $(k,\ell)$.
Now, we define $\mathbb F_{(3)}(-2d+3,0)$ as the quotient of $\mathbf{W}\times \mathbb C^2$ 
by the equivalence relation $\sim$ given by

$(x,y,z,t,\lambda_0,\lambda_1)\sim(x',y',z',t',\lambda'_0,\lambda'_1)$
$$\ \Leftrightarrow\ \exists \mu,\nu\in\CC^*,\ 
  (x',y',z',t',\lambda'_0,\lambda'_1)=(\mu x,\mu y,\mu z,\mu t,\mu^{2d-3}
  \nu\lambda_0,\nu\lambda_1)    .$$
We observe that  $H^0(\mathbb F_{(3)}(-2d+3,0))$ corresponds to the
set of $P\in \mathbb C[x,y,z,t,\lambda_0,\lambda_1]$ with homogeneous
multidegree $\mbox{multideg}$ defined above.

Now, since $F\in Sym^{d}(\mathbf{W}^\vee) $, $\alpha\in Sym^{d-1}(\mathbf{W}^\vee)$, $\beta\in Sym^{3d-4}(\mathbf{W}^\vee)$ and $\gamma\in Sym^{5d-7}(\mathbf{W}^\vee)$, 
we get that $F$ and $Q_{{\mathbf{S}},F}$ are in $H^0(\mathbb F_{(3)}(-2d+3,0))$. 
Therefore $\hat Z$ is a subvariety of $\mathbb F_{(3)}(-2d+3,0)$ since it can be rewritten:
$$\hat Z=\{(m,[\lambda_0:\lambda_1])\in \mathbb F_{(3)}(-2d+3,0)\ :\ F(\mathbf{m})=0\ \ \mbox{and}\ \ 
     Q_{{\mathbf{S}},F}(\mathbf{m},\lambda_0,\lambda_1)=0\}.$$
Since each coordinate of $\sigma$ is in $Sym^{2d-2}(\mathbf{W}^\vee)$, the map $\boldsymbol{\Phi}:\mathbf{W}\times\mathbb C^2\rightarrow\mathbf W$
given by
$$\boldsymbol{\Phi}(\mathbf{m},\lambda_0,\lambda_1):=\lambda_0 \cdot\mathbf{m}+\lambda_1\cdot\boldsymbol{\sigma}
(\mathbf{m})\ \ \in (C_{2(d-1),1})^4$$
defines a rational map $\Phi:\mathcal X\rightarrow\mathbb P^3$ with
$$\mathcal X=\{(m,[\lambda_0:\lambda_1])\in \mathbb F_{(3)}(-2d+3,0)\ :\ 
     Q_{{\mathbf{S}},F}(\mathbf{m},\lambda_0,\lambda_1)=0\}.$$
Let us denote by $B_{\Phi_{|\hat Z}}$ the set of base points of the map $\Phi_{|\hat Z}$, i.e.
$$B_{\Phi_{|\hat Z}}:=\{(m,[\lambda_0:\lambda_1])\in\hat Z\ :\ \boldsymbol{\Phi}(\mathbf{m},\lambda_0,\lambda_1)=0\} .$$
We consider the canonical projection $\pi_1:\mathbb F_{(3)}(-2d+3,0)\rightarrow \mathbb P^3$
(given by $\pi_1(m,[\lambda_0:\lambda_1])=m$).
\begin{nota}\label{ensB}
We write $\B:=\pi_1(B_{\Phi_{|\hat Z}})$.
\end{nota}
Observe that, for any $m\in\B$, there exists a unique
$[\lambda_0:\lambda_1]\in\mathbb P^1$
such that $\lambda_0\cdot\mathbf{m}+\lambda_1\cdot \boldsymbol{\sigma}(\mathbf{m})=0$.
This gives the following scheme
\begin{equation}
				\nonumber
				\begin{array}{ccccccc}
 				B_{\Phi_{|\hat Z}} & \hookrightarrow & \hat Z&\hookrightarrow&\mathcal X&\stackrel{\Phi}{\longrightarrow}&\mathbb P^3\\
 				{\pi_1}_{|B_\Phi}\downarrow 1:1&\ & {\pi_1}_{|\hat Z}\downarrow 2:1\ \  &\ 
                        &{\pi_1}_{|\mathcal X}\downarrow 2:1\ \  &\hfill&\hfill\\
                        \mathcal B& \hookrightarrow & \mathcal Z&\hookrightarrow& \mathbb P^3&\hfill&\hfill\\
				\end{array}
				\end{equation}
Therefore $\# B_\phi=\# \B$.
\begin{rqe}
The caustic by reflection $\Sigma_{S}(\mathcal Z)$ of $\mathcal Z$ from $S$ satisfies
$\Sigma_{S}(\Z)=\overline{\Phi(\hat Z)}\ \ \subseteq \mathbb P^3.$
\end{rqe}
Note that $\mathcal B\subseteq\mathcal M_{S,\mathcal Z}$
(with $\mathcal M_{S,Z}$ defined in Definition \ref{defiM}).
Due to the classical blowing-up theorem, we obtain the following result valid in the general case.
\begin{prop}\label{blowingup}
Assume that the set $\B$ is finite and that $\dim(\mathcal Z\cap
\mathcal M_{S,\mathcal Z})\le 1$.
Then there exists $\delta\in\mathbb N^*\cup\{\infty\}$ such that, for a
generic point $P\in \Sigma_{ S}(\mathcal Z)$, we have $\#[\pi_1 (\Phi_{|\hat Z}^{-1}(\{P\}))\setminus\mathcal B]=\delta$.
\end{prop}
\begin{proof}
Observe that, by hypothesis, the set $B_{\Phi_{|\hat Z}}$ is finite.
Now, applying the blowing-up result given in \cite[Example II-7.17.3]{Hartshorne}, we get the existence of
a variety $\widetilde{\hat Z}$ and of two morphisms $\pi:\widetilde{\hat Z}\rightarrow\hat Z$
and $\widetilde{{\Phi}}:\widetilde{\hat Z}\rightarrow\mathbb P^3$ such that
\begin{itemize}
\item $\pi$ defines an isomorphism from $\pi^{-1}(\hat Z\setminus B_\Phi)$ onto $\hat Z\setminus B_\Phi$,
\item On $\pi^{-1}(\hat Z\setminus B_\Phi)$, we have $\tilde\Phi=\Phi\circ\pi$,
\item $\widetilde{{\Phi}}(\widetilde{\hat Z})$ is the Zariski closure of $\Phi(\hat Z\setminus B_\Phi)$, i.e.
$\widetilde{{\Phi}}(\widetilde{\hat Z})=\Sigma_{S}(\Z)$,
\item $\dim(\widetilde{\hat Z})=2$,
\item $E:=\widetilde{\hat Z}\setminus \pi^{-1}(\hat Z\setminus B_\Phi)$ is a variety of dimension at most 1.
\end{itemize}
Let $\delta$ be the degree of the morphism $\widetilde{{\Phi}}$.
If $\delta=\infty$, then $\dim(\Sigma_{S}(\Z))\le\dim (\widetilde{{\Phi}}(\widetilde{\hat Z}))<2$.
Assume now that $\delta<\infty$. 
Since $\widetilde{{\Phi}}$ is a morphism, every point of 
$\widetilde{{\Phi}}(\widetilde{\hat Z})$ has $\delta$ preimages by $\widetilde{{\Phi}}$ in
$\widetilde{{\Phi}}(\widetilde{\hat Z})$.
Now, observe that $\dim(\Sigma_S(\Z))=2$ and that
$\dim(\widetilde{{\Phi}}(E))<2$. 
Therefore, a generic point of 
$\Sigma_{S}(\mathcal Z)$ is in 
$\Phi(\hat Z)\setminus [(\mathcal Z\cap\mathcal M_{S,\mathcal Z})\cup \widetilde{{\Phi}}(E)]$.
Let $P$ in this set. We have
\begin{eqnarray*}
\delta&=&\#\Phi_{|\mathcal Z}^{-1}(\{P\})\\
  &=&\#\{(m,[\lambda_0:\lambda_1])\in \mathcal Z\times\mathbb P^1 \ 
      :\ Q_{\mathbf{S},F}(\mathbf{m},\lambda_0,\lambda_1)=0\ \ 
      \mbox{and}\ \ \ \Pi(\lambda_0\cdot \mathbf{m}+\lambda_1\cdot\boldsymbol{\sigma}(\mathbf{m}))=P\}.
\end{eqnarray*}
Observe that, for $m\in\mathcal Z\cap\mathcal M_{S,\mathcal Z}$,
$\Phi(\pi_1^{-1}(\{m\}))=\{m\}$ by Definition \ref{defiM}.
Since $P\not\in \mathcal Z\cap\mathcal M_{S,\mathcal Z}$,
we know that $\Phi_{|\mathcal Z}^{-1}(\{P\})\cap \pi_1^{-1}(\mathcal M_{S,\mathcal Z})=\emptyset$. So, for any $m\in\pi_1(\Phi_{|\mathcal Z}^{-1}(\{P\}))$, there exists a unique $[\lambda_0:\lambda_1]\in\mathbb P^1$ such that $(m,[\lambda_0,\lambda_1])\in\Phi_{|\mathcal Z}^{-1}(\{P\}) $. Therefore $\delta=\#\pi_1(\Phi_{|\mathcal Z}^{-1}(\{P\}))
=\#[\pi_1(\Phi_{|\mathcal Z}^{-1}(\{P\}))\setminus\mathcal B]$ 
(since $\#[\pi_1(\Phi_{|\mathcal Z}^{-1}(\{P\}))\cap\mathcal B
\subset  \pi_1(\Phi_{|\mathcal Z}^{-1}(\{P\}))\cap\mathcal M_{S,\mathcal Z}=\emptyset$.
\end{proof}
\begin{lem}\label{dim>1}
If $\dim (Z\cap\mathcal M_{S,\mathcal Z})=2$. Then $\Sigma_S(\mathcal Z)=\overline{\mathcal Z\setminus(V(\beta,\gamma)\setminus V(\alpha))}$.
\end{lem}
So if $\dim (Z\cap\mathcal M_{S,\mathcal Z})=2$, then $\Sigma_S(\mathcal Z)=\mathcal Z$ except if $\mathcal Z\subset V(\beta,\gamma)
\setminus V(\alpha)$ and in this last case $\Sigma_S(\mathcal Z)=\emptyset$.
\begin{proof}[Proof of Lemma \ref{dim>1}]
Assume that $\dim (Z\cap\mathcal M_{S,\mathcal Z})=2$. Then $\mathcal Z\subset\mathcal M_{S,\mathcal Z}$.
So, due to Proposition \ref{B1}, we have $\mathcal Z\subset \base(\sigma)$. Now, due to Definition
\ref{deficaustique}, we have
$$\Sigma_S(\mathcal Z)=\overline{\{m\in\mathcal Z\ :\ \exists[\lambda_0:\lambda_1]\in\mathbb P^1,\ Q_{\mathbf{S},F}
  (\mathbf{m},\lambda_0,\lambda_1)=0,\ \lambda_0\ne 0 \}} $$
and the result follows.
\end{proof}
\section{Base points of $\Phi$}\label{base}
\begin{prop}
The base points of $\Phi_{|\hat Z}$ are the points $(m,[\lambda_0:\lambda_1])\in \hat Z$ satisfying one 
of the following
conditions:
\begin{enumerate}
\item $m\in V(F,\Delta_{\mathbf S}F,Q(\nabla F))$ (i.e. $m\in Sing(\mathcal Z)$ or $m$ is a
point of tangency of $\Z$ with an isotropic plane containing $S$)
and $\lambda_0=0$,
\item $t=F_x=F_y=F_z=0$ and $x^2+y^2+z^2=0$ (i.e. $m$ is a cyclic point with $\mathcal T_m\mathcal Z
=\mathcal H^\infty$) and $\lambda_0=0$,
\item $t=F_x=F_y=F_z=0$ (i.e. $\mathcal T_m\mathcal Z
=\mathcal H^\infty$) and $H_F=0$ and $\lambda_0=0$,
\item  $m=S\in\mathcal Z$ and $[\lambda_0:\lambda_1]$ is the unique
element of $\mathbb P^1$ such that $\lambda_0\cdot \mathbf{m}+\lambda_1Q(\nabla F)
\cdot \mathbf{S}=0$,
\item $m$ is a cyclic point (i.e. $m\in\mathcal C_\infty$),  
$[\lambda_0:\lambda_1]=[2\Delta_{\mathbf{S}}F(F_{xx}+F_{yy}+F_{zz}):2d-1]\ne[0:1]$ and 
$(F_{xx}+F_{yy}+F_{zz})\cdot\mathbf{m}=(2d-1)\boldsymbol{\kappa}(\nabla F)$.
\end{enumerate}
\end{prop}
\begin{proof}
Let us prove that any base point $(m,[\lambda_0:\lambda_1])$ has one of the form 
announced in the statement of the proposition (the converse being direct).
Let $(m,[\lambda_0:\lambda_1])\in B_{\Phi_{|\hat Z}}$. By definition of $\Phi$, 
we have $0=\lambda_0\cdot \mathbf{m}+\lambda_1\cdot \boldsymbol{\sigma}(\mathbf{m})$.
So $m\in\mathcal Z\cap\mathcal M_{S,\mathcal Z}$ and these $m$ have been determined in
Proposition \ref{B1}.
\begin{itemize}
\item Assume first that 
$\boldsymbol{\sigma}(\mathbf{m})=0$. Then 
the unique $[\lambda_0:\lambda_1]\in\mathbb P^1$ satisfying $\lambda_0\cdot\mathbf{m}
     +\lambda_1\cdot\boldsymbol{\sigma}(\mathbf{m})=0$ is $[\lambda_0:\lambda_1]=[0:1]$.
Hence $\lambda_0=0$, $\lambda_1\ne 0$ and so
$Q_{{\mathbf{S}},F}(\mathbf{m},\lambda_0,\lambda_1)=\gamma(\mathbf{m})\lambda_1^2. $

If $\Delta_{\mathbf{S}}F(\mathbf{m})=0$ and $Q(\nabla F(\mathbf{m}))=0$, then $\gamma(\mathbf{m})=0$.

Otherwise, according to Proposition \ref{baseR}, we have $F_x=F_y=F_z=0$
and so $t=0$. Now, we have $\gamma(\mathbf{m})=0$ if and only if $(x^2+y^2+z^2)t_0^2H_F=0$.
\item Assume now that $m=S$ and $\boldsymbol{\sigma}(\mathbf{m})\ne 0$.
Then $\Delta_{\mathbf{S}}F(\mathbf{m})=0$ and so
$\boldsymbol{\sigma}(\mathbf{m})=Q(\nabla F(\mathbf{m}))\cdot\mathbf{S}$. We consider the unique
$[\lambda_0,\lambda_1]$ such that $\lambda_0\cdot\mathbf{m}
     +\lambda_1\cdot\boldsymbol{\sigma}(\mathbf{m})=0$. Observe that $\lambda_0\ne 0$
and that  $\lambda_1\ne 0$. Since $\Delta_{\mathbf{S}}F(\mathbf{m})=0$, we have
$Q_{\mathbf{S},F}(\mathbf{m},\lambda_0,\lambda_1)= \beta(\mathbf{m})\lambda_0\lambda_1.$
But $\beta(\mathbf{m})=-2\hess_{F,\mathbf{m}}
(\mathbf{S},
     \boldsymbol{\sigma}(\mathbf{m}))=0$ due to $m=S=\sigma(m)$. 
\item Assume finally that $m\in\mathcal W$.
We have $\Delta_{\mathbf{S}}F(\mathbf{m})\ne 0$, $x^2+y^2+z^2=0$, $m=[F_x(\mathbf{m}):F_y(\mathbf{m}):F_z(\mathbf{m}):0]$, $t=0$ and $\boldsymbol{\sigma}(\mathbf{m})=
-2\Delta_{\mathbf {S}}F(\mathbf{m})\boldsymbol{\kappa}(\nabla F(\mathbf{m}))$. 
Since $t=0$, it follows that
$N_{\mathbf{S}}(\mathbf{m})=(x^2+y^2+z^2)t_0^2=0$ and so that $\gamma(\mathbf{m})=0$.
Let 
$[\lambda_0,\lambda_1]\in\mathbb P^1$ be such that $\lambda_0\cdot\mathbf{m}
     +\lambda_1\cdot\boldsymbol{\sigma}(\mathbf{m})=0$.
We observe that $\lambda_1\ne 0$ and $\lambda_0\ne 0$.
We have
$\alpha(\mathbf{m})\lambda_0^2=\Delta_{\mathbf{S}}F(\mathbf{m})\lambda_0^2 $
and
$$ -2\cdot\hess_{F,\mathbf m}(\mathbf{S},
     \boldsymbol{\sigma}(\mathbf{m}))\lambda_0\lambda_1=2\hess_{F,\mathbf{m}}(\mathbf{S},
     \mathbf{m})\lambda_0^2=2(d-1)\Delta_{\mathbf{S}}F(\mathbf{m}),$$
since $(d-1)F_w=xF_{xw}+yF_{yw}+zF_{zw}+tF_{tw}$ for every $w\in\{x,y,z\}$.
Hence we have
$$ Q_{\mathbf{S},F}(\mathbf{m},\lambda_0,\lambda_1)=(2d-1)
\Delta_{\mathbf{S}}F(\mathbf{m})\lambda_0^2-2(\Delta_{\mathbf{S}}F(\mathbf{m}))^2
    (F_{xx}(\mathbf{m})+F_{yy}(\mathbf{m})+F_{zz}(\mathbf{m}))\lambda_0\lambda_1.$$
Hence $Q_{\mathbf{S},F}(\mathbf{m},\lambda_0,\lambda_1)=0$ if and
only if $\lambda_0/\lambda_1=2\Delta_{\mathbf{S}}F(\mathbf{m})(F_{xx}(\mathbf{m})+F_{yy}(\mathbf{m})+F_{zz}(\mathbf{m}))/(2d-1)$. We conclude by using $\lambda_0\cdot\mathbf{m}+\lambda_1\cdot\boldsymbol{\sigma}(\mathbf{m})=0$ and the formula obtained for $\boldsymbol{\sigma}(
\mathbf{m})$.
\end{itemize}
\end{proof}
\begin{coro}\label{B}
A point $m$ is in $\mathcal B$ if and only if it satisfies one of the following conditions:
\begin{enumerate}
\item $m\in \mathcal B_0=V(F,\Delta_{\mathbf{S}}F,Q(\nabla F))$,
i.e. $m$ is a singular point of $\mathcal Z$ or $m$
is a point of tangency of $\mathcal Z$ with an isotropic plane containing $S$ (see also \refeq{equiv}),
\item $m$ is a point of tangency of $\mathcal Z$ with $\mathcal H^\infty$ and $m$ lies on the umbilical curve $\mathcal C_\infty$,
\item $m$ is a point of tangency of $\mathcal Z$ with $\mathcal H^\infty$ and $m$ lies in the hessian surface of $\mathcal Z$,
\item $m=S\in\mathcal Z$,
\item $m$ lies on $\mathcal C_\infty$ and
$(F_{xx}+F_{yy}+F_{zz})\cdot\mathbf{m}=(2d-1)\boldsymbol{\kappa}(\nabla F)$.
\end{enumerate}
This can be summarized in the following formula
$$\mathcal B=\mathcal Z\cap[V(\Delta_{\mathbf S}F,Q(\nabla F))
\cup\{S\}\cup V(H_F\cdot Q,\boldsymbol{\kappa}(\nabla F))_\infty
   \cup \mathcal G_\infty], $$
with $\mathcal G_\infty=\{m\in\mathcal C_\infty\ :\ 
(F_{xx}+F_{yy}+F_{zz})\cdot\mathbf{m}=(2d-1)\boldsymbol{\kappa}(\nabla F)\} $.
\end{coro}
\begin{rqe}\label{Bgene}
The set  $\mathcal B$ is never empty. Except (iv), the forms of the base points
are very similar to the base points of the caustic map of planar curves (see \cite{fredsoaz1}).

For a general $(\mathcal Z,S)$, the set $\mathcal B$ consists of the
points at which $\mathcal Z$ admits an isotropic tangent plane containing $S$, i.e. $\mathcal B=\mathcal B_0=V(F,\Delta_{\mathbf S}F,Q(\nabla F))$, and in general $\mathcal Z$ has no singular 
point and $\mathcal B_0\cap\mathcal H^\infty=\emptyset$.
In this case $\mathcal B$ is fully interpreted by
\refeq{equiv}.
\end{rqe}
Let us study the base points when $\mathcal Z$ is
the paraboloid $V(x^2+y^2-2zt)$ (see Example
\ref{exaparaboloid}).
\begin{prop}[Paraboloid]\label{baseparaboloide}
Let $\mathcal Z=V(x^2+y^2-2zt)$
and any $S\in\mathbb P^3\setminus\{F_1,F_2\}$.
Then $\mathcal B=V(F,\Delta_{\mathbf S}P,Q(\nabla F))\cup(\{S\}\cap\mathcal Z)$ and its points
are the following ones:
\begin{enumerate}
\item the point $S$ if $S$ is in $\mathcal Z$,
\item the points $[1:\pm i:0:0]$ if $x_0=y_0=0$,
\item the point $[t_0:it_0:x_0+iy_0:0]$ and $[t_0:-it_0:x_0-iy_0:0]$ if $t_0\ne 0$,
\item the point $F_1[0:0:1:0]$ if $t_0=0$ and $x_0^2+y_0^2\ne 0$,
\item the points of the form $[ux_0:uy_0:z:0]$ (with $[u:z]\in\mathbb P^1$)
if $x_0^2+y_0^2=0$, $t_0=0$,
\item the points $m_{1}$ and $m_{-1}$ if $x_0^2+y_0^2\ne 0$, with
$$m_\varepsilon:=[x_\varepsilon:y_{\varepsilon}:-\frac 12:1] ,$$
$$x_{\varepsilon}=\frac{x_0(z_0-\frac {t_0}2)+i\varepsilon y_0
\sqrt{x_0^2+y_0^2+(z_0-\frac {t_0}2)^2}} {x_0^2+y_0^2} $$
and
$$y_{\varepsilon}=\frac{y_0(z_0-\frac {t_0}2)-i\varepsilon x_0
\sqrt{x_0^2+y_0^2+(z_0-\frac {t_0}2)^2}} {x_0^2+y_0^2}.$$
\item the point $\left[\frac{z_0-\frac {t_0}2}{2x_0}:\frac{z_0-\frac {t_0}2}{2y_0}:-\frac 12:1\right]$
if $x_0^2+y_0^2=0$, $x_0\ne 0$ and 
$z_0-\frac{t_0}2\ne 0$.
\end{enumerate}
\end{prop}
\begin{proof}
We have $\alpha=\Delta_{\mathbf S}F=xx_0+yy_0-tz_0-zt_0$, 
$$\tilde\beta=(x^2+y^2+t^2)(x_0^2+y_0^2-2t_0z_0)-2
     (tz_0+zt_0+t_0t)\Delta_{\mathbf S}F $$
and
$\gamma=
4((x_0t-xt_0)^2+(y_0t-yt_0)^2+(z_0t-zt_0)^2)\Delta_{\mathbf S}F$. 
First we observe that there is no point of $\mathcal Z$ satisfying (5) of Corollary \ref{B}. Assume that $m$ is such a point. 
We have $t=0$, so $x^2+y^2=0$ and
$z=0$ (since $x^2+y^2+ z^2=0$). Now, using $(F_{xx}+F_{yy}+F_{zz})\mathbf{m}=(2d-1)\boldsymbol{\kappa}(\nabla F)$, we obtain $2x=3x$ and $2y=3y$, which
implies $x=y=0$ which contradicts $z=t=0$.

Now we prove that $\mathcal Z\cap V(H_F\cdot Q,\boldsymbol{\kappa}
(\nabla F))\cap\mathcal H^\infty=\emptyset$.
Assume that $m\in\mathcal Z\cap V(H_F\cdot Q,\boldsymbol{\kappa}
(\nabla F))\cap\mathcal H^\infty$. Due to $F_x(\mathbf{m})=F_y(\mathbf{m})
=F_z(\mathbf{m})=0$, we obtain $x=y=t=0$. Since $H_F=-1$, we have
$0=x^2+y^2+z^2$ and so $z=0$.

It remains to identify the points of 
$V(F,\Delta_SF,Q(\nabla F))$. Let $m[x:y:z:t]$ be a point of this set.
Then $x^2+y^2=2zt$, $0=x_0x+y_0y-z_0t-t_0z$,
$0=F_x^2+F_y^2+F_z^2=x^2+y^2+t^2$. So $-t^2=x^2+y^2=2zt$.

If $x_0=y_0=0$, since $t_0\ne 0$, we obtain $x^2+y^2+t^2=0=x^2+y^2-2zt$
and $z=-z_0t/t_0$. If $t=0$, we obtain $x^2+y^2=0$ and $z=0$. This gives (2).
If $t=1$, from $t^2=-2zt$, it comes $z=-1/2$ and so $t_0=2z_0$,
which contradicts $S\ne F_2$.

From now on, we assume that $(x_0,y_0)\ne 0$.

Assume first that $t=0$. Then we have $x^2+y^2=0$ and $0=\Delta_{\mathbf{S}}F
=x_0x+y_0y-t_0z$. Let $\varepsilon\in\{\pm 1\}$ such that $y=i\varepsilon x$.
We have $0=x_0x+i\varepsilon y_0x-t_0z$.
If $t_0=0$, this becomes $0=x(x_0+i\varepsilon y_0)$ and so either $x=0$
or $x_0+i\varepsilon y_0=0$. This gives (4) and (5).
If $t_0\ne 0$ and if $x_0+i\varepsilon y_0=0$, we obtain $z=0$ and $y=i\varepsilon x$, so $m=|1:i\varepsilon:0:0]$.
If $t_0\ne 0$ and if $x_0+i\varepsilon y_0\ne 0$, we obtain $x=t_0z/(x_0+i\varepsilon y_0)$, so $m=[t_0:i\varepsilon t_0:x_0+i\varepsilon y_0:0]$. This gives (3).

Assume now that $t=1$. We have $-1=x^2+y^2=2z$ and so $z=-1/2$.
Since $x^2+y^2=-1$, we consider
$\varepsilon\in\{\pm 1\}$ be such that $y=\varepsilon i\sqrt{1+x^2}$.
We have
$$0=x_0x+y_0y-z_0t-t_0z=x_0x+y_0\varepsilon i\sqrt{1+x^2}-z_0+\frac{t_0}2$$
and so $-y_0\varepsilon i\sqrt{1+x^2}=x_0x-z_0+\frac{t_0}2 $
and we obtain
$-y_0^2(1+x^2)=(x_0x-z_0+\frac{t_0}2)^2$
and so
$$0=(x_0^2+y_0^2)x^2-2xx_0\left(z_0-\frac{t_0}2\right)+\left(z_0-\frac{t_0}2\right)^2+y_0^2 .$$
This gives $m\in\{m_1,m_{-1}\}$ if $x_0^2+y_0^2\ne 0$, with 
$m_\varepsilon:=[x_\varepsilon:y_{\varepsilon}:-\frac 12:1]$,
$$x_{\varepsilon}=\frac{x_0(z_0-\frac {t_0}2)+
\varepsilon\sqrt{(x_0(z_0-\frac {t_0}2))^2-(x_0^2+y_0^2)(y_0^2+(z_0-\frac {t_0}2)^2)}}{x_0^2+y_0^2} $$
and
$$y_{\varepsilon}=\frac{y_0(z_0-\frac {t_0}2)-
\varepsilon\sqrt{(y_0(z_0-\frac {t_0}2))^2-(x_0^2+y_0^2)(x_0^2+(z_0-\frac {t_0}2)^2)}}{x_0^2+y_0^2},$$
and so (6).
Now, we assume moreover that $x_0^2+y_0^2=0$, we obtain
$$0=-2xx_0\left(z_0-\frac{t_0}2\right)+\left(z_0-\frac{t_0}2\right)^2+y_0^2 .$$
If $x_0^2+y_0^2=0$ and $z_0-\frac{t_0}2\ne 0$, then we obtain
$x=\frac{z_0-\frac {t_0}2}{2x_0}$ and $y=\frac{z_0-\frac {t_0}2}{2y_0}$,
since $x_0\ne 0$ and $y_0\ne 0$. This gives (7).
If $x_0^2+y_0^2=0$ and $z_0-\frac{t_0}2=0$, we obtain $y_0^2=0$ and so 
$S=F_2$.
\end{proof}
\begin{rqe} Let $\mathcal Z=V((x^2+y^2-2zt)/2)$ and $S\in\mathbb P^3\setminus\{F_1,F_2\}$.
Observe that $\#\mathcal B<\infty$
except if $S\in\mathcal Z_\infty$.
\end{rqe}
In the particular case where $S\in\mathcal Z_\infty$, we have the following.
\begin{prop}\label{cas1}
Let $\mathcal Z=V((x^2+y^2-2zt)/2)$ and 
$S=[1:\varepsilon i:z_0:0]$ with $\varepsilon\in\{\pm 1\}$, 
then $\Sigma_S(\mathcal Z)$ is the curve of equations $ z_0i\varepsilon x+z_0y+i\varepsilon t,
    (z-t/2)^2+x^2+y^2$.
\end{prop}
\begin{proof}
In this case, we have $\alpha=x+i\varepsilon y-tz_0$,  $\tilde\beta=-2\alpha tz_0$ and $\gamma=4\alpha z_0^2t^2$.
So (\ref{formequadratique}) becomes $\alpha(\lambda_0+2tz_0\lambda_1)^2=0$.
Hence $\Sigma_S(\mathcal Z)$ is the Zariski closure of the image of
$\mathcal Z$ by the rational map given by $\boldsymbol{\psi}(\mathbf{m})=
-2tz_0\cdot\mathbf{m}+\boldsymbol{\sigma}(\mathbf{m}).$
On $\mathcal Z$, using the fact that $2tz=x^2+y^2$, this rational map
can be rewritten
$\boldsymbol{\psi}(\mathbf{m})=\left(\begin{array}{c}-(x+i\varepsilon y)^2+t^2\\i\varepsilon((x+i\varepsilon y)^2+t^2)\\-z_0t^2+2(x+i\varepsilon y)t\\
-2t^2z_0\end{array}\right).$ To conclude observe $\boldsymbol{\psi}$
only depends on $(x+i\varepsilon y,t)$ and that if $(X,Y,Z,T)=\boldsymbol{\psi}(x,y,z,t)$, then $2t^2=X-\varepsilon Y$ 
and $2(x+i\varepsilon y)^2=-X-i\varepsilon Y$.

\end{proof}
\section{Reflected Polar curves}\label{polar}
Let $H\in Pic(\mathbb P^3)$ be the hyperplane class. We will identify $\pi_{1,*}(\Phi^*H^2)\in A_2(\mathbb P^3)$
with the class of sets $\mathcal P_{A,B}\subseteq \mathbb P^3$ defined as follows.
\begin{defi}\label{defirefpolar}
For any $A,B\in \mathbf{W}^\vee$, we define the set 
$\mathcal D_{A,B}:=V(A,B)\subseteq \PP^3$ and the 
\emph{reflected polar} $\mathcal P_{A,B}$ by
$$\mathcal P_{A,B}=\pi_1(\Phi^{-1}(\mathcal D_{A,B}))\cup \pi_1(Base(\Phi)), $$
i.e. $\mathcal P_{A,B}$ corresponds to the following set:
$$ \{m\in\mathbb P^3 :\ \exists [\lambda_0,\lambda_1]\in\PP^1,\ A(\lambda_0\mathbf{m}+
   \lambda_1\boldsymbol{\sigma}(\mathbf{m}))=0,\ B(\lambda_0\mathbf{m}+
   \lambda_1\boldsymbol{\sigma}(\mathbf{m}))=0,\ 
     \ Q_{{\mathbf{S}},F}(\mathbf{m},\lambda_0,\lambda_1)=0\}. $$
\end{defi}
\begin{prop}\label{K1K2K3}
For  generic $A,B$ in $\mathbf{W}^\vee$,
$\mathcal D_{A,B}$ is a line and
$\mathcal P_{A,B}=V(K_1,K_2,K_3)$, with
$$K_1(\mathbf{m}):=A(\boldsymbol{\sigma}(\mathbf{m}))B(\mathbf{m})-A(\mathbf{m})B(\boldsymbol{\sigma}(\mathbf{m})),\ \  K_2(\mathbf{m}):=
    Q_{{\mathbf{S}},F}(\mathbf{m},-A(\boldsymbol{\sigma}(\mathbf{m})),
A(\mathbf{m})),$$
$$K_3(\mathbf{m}):=
Q_{{\mathbf{S}},F}(\mathbf{m},-B(\boldsymbol{\sigma}
(\mathbf{m})),B(\mathbf{m})). $$
\end{prop}
\begin{proof}
Recall that $\mathcal M_{S,\mathcal Z}$ has been defined in Definition \ref{defiM}.
Assume that $\mathcal D_{A,B}$ is a line that does not correspond to any line
$(m\, \sigma(m))$ for $m\in\mathbb P^3\setminus  \mathcal M_{S,\mathcal Z}$ 
(this is true for a generic $(A,B)$ in $(\mathbf{W}^\vee)^2$).
Hence
\begin{equation}\label{ensvide}
V(A,B,A\circ\boldsymbol{\sigma},B\circ\boldsymbol{\sigma})=\mathcal M_{S,\mathcal Z}\cap \mathcal D_{A,B}.
\end{equation}
Let $m\in\mathcal P_{A,B}$ and $[\lambda_0:\lambda_1]\in\mathbb P^1$ 
be such that $A(\boldsymbol{\Phi}(\mathbf{m},\lambda_0,\lambda_1))=0
    =B(\boldsymbol{\Phi}(\mathbf{m},\lambda_0,\lambda_1)
 )=0$ and $Q_{\mathbf{S},F}(\mathbf{m},\lambda_0,\lambda_1)=0$. 
This implies that
$$\lambda_0\cdot A(\mathbf{m})+\lambda_1\cdot A(\boldsymbol{\sigma}(\mathbf{m}))=0=
 \lambda_0\cdot B(\mathbf{m})+\lambda_1\cdot B(\boldsymbol{\sigma}(\mathbf{m})).$$
Therefore $(-A(\boldsymbol{\sigma}(\mathbf{m})),A(\mathbf{m}))$ and
$(-B(\boldsymbol{\sigma}(\mathbf{m})),B(\mathbf{m}))$
are proportional to $(\lambda_0,\lambda_1)$.
But, since $Q_{{\mathbf{S}},F}(\mathbf{m},\lambda_0,\lambda_1)=0$, we conclude that $m$ is in $V(K_1,K_2,K_3)$.

Conversely, assume now that $m$ is a point of $V(K_1,K_2,K_3)$. 
Due to (\ref{ensvide}), we have
$$  m\not\in V(A,B,A\circ\boldsymbol{\sigma},B\circ\boldsymbol{\sigma})\ \ \mbox{or}\ \ m\in\mathcal M_{S,\mathcal Z}\cap \mathcal D_{A,B}.$$
Assume first that $m\not\in V(A,B,A\circ\boldsymbol{\sigma},B\circ\boldsymbol{\sigma})$, then
$ (-A(\boldsymbol{\sigma}(\mathbf{m})),A(\mathbf{m}))$
and $(-B(\boldsymbol{\sigma}(\mathbf{m})),B(\mathbf{m}))$
are proportional and at least one is non null. Let
$[\lambda_0:\lambda_1]$ be the corresponding point in $\PP^1$.
we have
$A(\lambda_0\cdot\mathbf{m}+\lambda_1\cdot\boldsymbol{\sigma}(\mathbf{m}))=0$, 
$ B(\lambda_0\cdot\mathbf{m}+\lambda_1\cdot\boldsymbol{\sigma}(\mathbf{m}))=0$, and 
$Q_{{\mathbf{S}},F}(\mathbf{m},\lambda_0,\lambda_1)=0$. So
$m\in\mathcal P_{A,B}$.

Assume finally that $m\in\mathcal M_{S,\mathcal Z}\cap \mathcal D_{A,B}$,
then there exists $[\lambda:\mu]\in\mathbb P^1$ such that
$\lambda\cdot\mathbf{m}+\mu\cdot\boldsymbol{\sigma}(\mathbf m)=0$, $A(\mathbf m)=B(\mathbf m)=0$. We also have
$A(\boldsymbol{\sigma}(\mathbf m))=B(\boldsymbol{\sigma}(\mathbf m))=0$ and so $m\in\mathcal P_{A,B}$.
\end{proof}
\begin{nota}\label{BSZ}
We write $B_{S,\mathcal Z}$ for the set
of points $m\in\mathbb P^3$ for which $Q_{{\mathbf{S}},F}(\mathbf{m},\lambda_0,\lambda_1)=0$
in $\mathbb C[\lambda_0,\lambda_1]$.
\end{nota}
Observe that, for $m\in\mathbb P^3\setminus B_{S,\mathcal Z}$, there are at most two $[\lambda_0:\lambda_1]\in\mathbb P^1$
such that $Q_{{\mathbf{S}},F}(\mathbf{m},\lambda_0,\lambda_1)=0$, and so $\#(\mathcal X\cap\pi_1^{-1}(m))\le  2$.
\begin{rqe}\label{rqeBSZ}
According to the expressions of $\sigma$, $\alpha$, $\beta$ and $\gamma$, we have
$$ B_{S,\mathcal Z}=V(\alpha,\beta,\gamma)=V(\Delta_{\mathbf{S}}F, \hess_F(\mathbf{S},\mathbf{S})\cdot Q(\nabla F)).$$ 
\end{rqe}
We observe that $\dim B_{S,\mathcal Z}\ge 1$.
Observe that $(\mathcal M_{S,\mathcal Z}\cup B_{S,\mathcal Z})$
is the set of $m\in\mathbb P^3$ such that $\Phi(\pi^{-1}(\{m\}))
=\Pi(Vect(\mathbf{m},\boldsymbol{\sigma}(\mathbf{m})))$. When
$m\in \mathcal M_{S,\mathcal Z}$, $\Phi(\pi^{-1}(\{m\}))=\{m\}$
and when $m\in B_{S,\mathcal Z}$, $\Phi(\pi^{-1}(\{m\}))=\mathcal R_m$.
Recall that $\mathcal B=\pi_1(B_{\Phi_{|\hat Z}})$.
\begin{prop}\label{degreepolar}
Assume that $\# \mathcal B<\infty$ and $\mathcal Z\ne\mathcal H^\infty$.
Then, for generic $A,B\in\mathbf{W}^\vee$, $\dim\mathcal P_{A,B}=1$ and 
$\deg\mathcal P_{A,B}=(d-1)(10d-9)$.
\end{prop}
\begin{proof}
As in the proof of the preceding proposition, we consider generic $(A,B)\in (\mathbf{W}^\vee)^2$ such that
(\ref{ensvide}) holds.
Recall that $\mathcal P_{A,B}=V(K_1,K_2,K_3)$.
First, we observe that $K_1\in Sym^{2d-1}(\mathbf{W}^\vee)$
whereas $K_2,K_3\in 
Sym^{5(d-1)}(\mathbf{W}^\vee)$. 
Now, if $m$ is a point of $\mathbb P^3\setminus 
V(A,A\circ\boldsymbol{\sigma})$, then the following equivalence holds true
$m\in\mathcal P_{A,B}\ \ \Leftrightarrow\ \  m\in V(K_1,K_2)$
and that, if $m$ is a point of $\mathbb P^3\setminus
V(B,B\circ\boldsymbol{\sigma})$, then 
$m\in\mathcal P_{A,B}\Leftrightarrow m\in V(K_1,K_3).$
Therefore $\dim\mathcal P_{A,B}\in\{1,2\}$.
\begin{itemize}
\item Let us prove that $\dim\mathcal P_{A,B}=1$.
Assume first that $\dim(\Sigma_S(\mathcal Z))\le 1$. 
Then, for generic $(A,B)\in(\mathbf{W}^\vee)^2$, we have $\Sigma_{\mathcal S}(\mathcal Z)\cap \mathcal D_{A,B}=\emptyset$.
Therefore, $\pi_1(\Phi_{|\hat Z}^{-1}(\mathcal D_{A,B}))=\emptyset$ and so $\mathcal Z\cap \mathcal P_{A,B}=\mathcal B$
is a finite set, which implies that $\dim\mathcal P_{A,B}\le 1$.

Assume now that $\dim(\Sigma_S(\mathcal Z))=2$.
Let us consider a generic $(A,B)\in(\mathbf{W}^\vee)^2$ such that $\#(\Sigma_{S}(\mathcal Z)
    \cap \mathcal D_{A,B})<\infty$ and such that, for every $P\in \Sigma_{S}(\mathcal Z)
    \cap \mathcal D_{A,B}$, we have
$\#[\pi_1(\Phi_{|\hat Z}^{-1}(\{P\}))\setminus\mathcal B]=\delta,$
(see Proposition \ref{blowingup}). This
implies that $\#\pi_1(\Phi_{|\hat Z}^{-1}(\mathcal D_{A,B}))<\infty$ and so 
$\#(\mathcal Z\cap \mathcal P_{A,B})<\infty$ (since $\#\mathcal B<\infty$). 
Hence $\dim\mathcal P_{A,B}\le 1$
since $\dim\mathcal Z=2$.
\item Let $(A,B)$ as above. Since $\dim\mathcal P_{A,B}=1$,
$\deg \mathcal{P}_{A,B}$ corresponds to $\#(\mathcal P_{A,B}\cap \mathcal{H})$ for a generic plane $\mathcal H$ in $\mathbb P^3$.
Due to Corollary \ref{B}, we have $\#V(F,\Delta_{\mathbf S}F,Q(\nabla F))\le\#\mathcal B<\infty$. So
$\dim V(\Delta_{\mathbf S}F,Q(\nabla F))=1$. Moreover we have $\mathcal Z\ne\mathcal H^\infty$.
So, due to Lemma \ref{dim3}, we conclude that $\dim\mathcal M_{S,\mathcal Z}<3$
and we assume that $\# (\mathcal M_{S,\mathcal Z}\cap \mathcal D_{A,B})<\infty$.
Now, since (\ref{ensvide}) holds, we conclude that
$\dim V(A,A\circ\boldsymbol{\sigma})=1=\dim V(B,B\circ\boldsymbol{\sigma})$.

Since $\dim\mathcal P_{A,B}=1$, we conclude that $\dim V(K_1,K_2)=\dim V(K_1,K_3)=1$.
Moreover, for a generic $(A,B)\in(\mathbf{W}^\vee)^2$, we have
$$\# E<\infty, \ \mbox{with}\ E:=V(A,A\circ\boldsymbol{\sigma},K_3)
        \cup V(B,B\circ\boldsymbol{\sigma},K_2)<\infty .$$
Let us explain how we get $\#V(A,A\circ\boldsymbol{\sigma},K_3)<\infty$.
We recall that, due to \eqref{ensvide}, $\# V(A,A\circ\sigma,B,B\circ\sigma)<\infty$.
We write $E_1:=V(A,A\circ\boldsymbol{\sigma},K_3)\setminus V(B,B\circ\sigma)$ to simplify notations. 
\begin{itemize}
\item First we observe that $E_1\cap\mathcal M_{S,\mathcal Z}\subset V(A)\cap\mathcal M_{S,\mathcal Z}$
which is finite for a generic $A\in \mathbf{W}^\vee$ since $\dim\mathcal M_{S,\mathcal Z}<3$.
\item Second we observe that $E_1\cap(B_{S,\mathcal Z}\setminus\mathcal M_{S,\mathcal Z})=\emptyset$
for a generic $A\in \mathbf{W}^\vee$.
Indeed this set is contained in
$V(A,A\circ\boldsymbol{\sigma},\Delta_{\mathbf S}F)\setminus V(Q(\nabla F))$. For $m$ in this set,
we have $\sigma(m)=S$. So, we just have to take $A$ such that $A(\mathbf S)\ne 0$.
\item Third we observe that $\# E_1\setminus B_{S,\mathcal Z}<\infty$ for generic
$A,B\in \mathbf{W}^\vee$.
Indeed, $\dim V(A,A\circ\sigma)=1$ (due to \eqref{ensvide}) and, for any $m\in V(A,A\circ\boldsymbol{\sigma})\setminus B_{S,\mathcal Z}$,
there are at most two $[\lambda_0:\lambda_1]\in\mathbb P^1$ such that 
$Q_{{\mathbf{S}},F}(m,\lambda_0,\lambda_1)=0$. So, for a generic $B\in \mathbf{W}^\vee$, we have
$$\#\{(m,[\lambda_0:\lambda_1])\in \mathcal X :\ m\in V(A,A\circ\boldsymbol{\sigma})\setminus B_{S,\mathcal Z},\ 
     B(\Phi(m,[\lambda_0,\lambda_1]))=0\}<\infty.$$
Now, if $m\in E_1\setminus B_{S,\mathcal Z}$, 
$[-B(\boldsymbol{\sigma}(\mathbf{m})):B(\mathbf{m})]\in\mathbb P^1$ and, due to $K_3(\mathbf{m})=0$, we observe that
$M:=(m,[-B(\boldsymbol{\sigma}(\mathbf{m})):B(\mathbf{m})])$ is in $\mathcal X$ and $B(\Phi(M))=0$.
So $\# E_1\setminus B_{S,\mathcal Z}<\infty$.
\end{itemize}
Let us prove that $\deg(\mathcal P_{S,\mathcal Z,A,B})=(d-1)(10d-9)$
for a generic $(A,B)\in(\mathbf W^\vee)^2$.
We consider generic $A$ and $B$ such that:
$\deg A= 1$, $\deg A\circ\sigma=2d-2$, $V( A, A\circ\sigma,
 B,B\circ\sigma)<\infty$ and
$\# V(A,A\circ\sigma,K_3)<\infty$.
We observe that the reflected polar curve corresponds to $V(K_1,K_2)$
outside $V(A, A\circ\sigma)$ and that
the polar curve coincide with $V(K_3)$ on $V( A, A\circ\sigma)$ (since $V(A,A\circ\sigma)$ is contained in $V(K_1,K_2)$).
We consider a generic plane $\mathcal H$ in $\mathbb P^3$, 
such that $\mathcal H\cap V( A, A\circ\sigma,
 B,B\circ\sigma)=\emptyset $ and $\mathcal H\cap V( A, A\circ\sigma,K_3)=\emptyset$.
Since  $\mathcal H\cap V( A, A\circ\sigma,K_3)=\emptyset$, we have
\begin{eqnarray*}
\deg(\mathcal P_{S,\mathcal Z,A,B})&=&\deg V(K_1,K_2)
    -\sum_{m\in V( A,A\circ\sigma)\cap\mathcal H}
       i_m(\mathcal H,V(K_1,K_2))\\
&=&5(d-1)(2d-1)-\sum_{m\in V(A,A\circ\sigma)\cap\mathcal H}
       i_m(\mathcal H,V(K_1,K_2)).
\end{eqnarray*}
Now, we prove that
\begin{equation}\label{double}
\forall m\in V(A,A\circ\sigma)\cap\mathcal H,\ \ i_m(\mathcal H,V(K_1,K_2))=2i_m(\mathcal H,V(A,A\circ\sigma)).
\end{equation}
Let $m\in V(A,A\circ\sigma)\cap\mathcal H$.
If $B(m)\ne 0$, then 
\begin{eqnarray*}
i_m(\mathcal H,V(K_1,K_2))&=&i_m(\mathcal H,K_1,K_2)=i_m(\mathcal H,K_1,B^2\cdot K_2)\\
&=& i_m(\mathcal H,K_1,\alpha(B\cdot A\circ\sigma)^2-\beta\cdot A\cdot B\cdot (B\cdot A\circ\sigma)+\gamma\cdot A^2\cdot B^2)\\
&=& i_m(\mathcal H,K_1,\alpha(A\cdot B\circ\sigma)^2-\beta\cdot A\cdot B\cdot (A\cdot B\circ\sigma)+\gamma\cdot A^2\cdot B^2),
\end{eqnarray*}
(since $B\cdot A\circ\sigma=A\cdot B\circ\sigma$ on $V(K_1)$) and so
\begin{eqnarray*}
i_m(\mathcal H,V(K_1,K_2))&=&i_m(\mathcal H,K_1,A^2\cdot K_3)\\
&=&i_m(\mathcal H,K_1,A^2)\ \ \ \mbox{since}\ K_3(m)\ne 0\\
&=&i_m(\mathcal H,A\circ\sigma\cdot B-B\circ\sigma\cdot A,A^2)\\
&=&2\, i_m(\mathcal H,A\circ\sigma\cdot B-B\circ\sigma\cdot A,A)\\
&=&2\, i_m(\mathcal H,A\circ\sigma\cdot B,A)\\
&=&2\, i_m(\mathcal H,A\circ\sigma,A)\ \ \ \mbox{since}\ B(m)\ne 0.
\end{eqnarray*}
Analogously, If $B(\sigma(m))\ne 0$, then 
\begin{eqnarray*}
i_m(\mathcal H,V(K_1,K_2))&=&i_m(\mathcal H,K_1,K_2)=i_m(\mathcal H,K_1,B^2\circ\sigma\cdot K_2)\\
&=& i_m(\mathcal H,K_1,\alpha \cdot B^2\circ\sigma\cdot A^2\circ\sigma-\beta\cdot A\circ\sigma.B\circ\sigma\cdot (A\cdot B\circ\sigma)+\gamma\cdot A^2\cdot B^2\circ\sigma)\\
&=& i_m(\mathcal H,K_1,\alpha(A\circ\sigma\cdot B\circ\sigma)^2-\beta\cdot A\circ\sigma\cdot B\circ\sigma\cdot (B\cdot A\circ\sigma)+\gamma\cdot B^2\cdot A^2\circ\sigma)\\
&=&i_m(\mathcal H,K_1,A^2\circ\sigma\cdot K_3)
=i_m(\mathcal H,K_1,A^2\circ\sigma)\\
&=&i_m(\mathcal H,A\circ\sigma\cdot B-B\circ\sigma\cdot A,A^2\circ\sigma)\\
&=&2\, i_m(\mathcal H,A\circ\sigma\cdot B-B\circ\sigma\cdot A,A\circ\sigma)=2\, i_m(\mathcal H,A,A\circ\sigma).
\end{eqnarray*}
Hence we proved (\ref{double}) and, for a generic plane $\mathcal H$, we have
\begin{eqnarray*}
\deg(\mathcal P_{S,\mathcal Z,A,B})
&=&5(d-1)(2d-1)-\sum_{m\in V(A,A\circ\sigma)\cap\mathcal H}
       i_m(\mathcal H,V(K_1,K_2))\\
&=&5(d-1)(2d-1)-2\sum_{m\in V(A,A\circ\sigma)\cap\mathcal H}
       i_m(\mathcal H,V(A,A\circ\sigma))\\
&=&5(d-1)(2d-1)-2\deg(A)\deg(A\circ\sigma)\\
&=&5(d-1)(2d-1)-4(d-1)=(d-1)(10d-9).
\end{eqnarray*}
\end{itemize}
\end{proof}
\section{A formula for the degree of the caustic}\label{degree}
Recall that $\mathcal B$ has been completely described in
Corollary \ref{B} (see also Remark \ref{Bgene} for the general case).
We refer to Definition \ref{defiM} and Proposition \ref{B1}
for $\mathcal M_{S,\mathcal Z }$ and to Notation \ref{BSZ} and
Remark \ref{rqeBSZ} for $\mathcal B_{S,\Z}$.
Observe that $\dim\mathcal M_{S,\mathcal Z}\ge 1$ since $\base(\sigma)\subseteq \mathcal M_{S,\mathcal Z}$.
\begin{thm}\label{lemmefondamental}
We assume that $\#\mathcal B<\infty$.

If $\dim(\Sigma_{S}(\Z))<2$, then for a generic $(A,B)\in(\mathbf{W}^\vee)^2$, we have
$$0=d(d-1)(10d-9) -\sum_{m\in \B}i_m(\mathcal Z,\mathcal P
     _{A,B}).$$

If $\dim(\Sigma_{S}(\Z))=2$, $\dim  (\mathcal Z\cap \mathcal M_{S,\mathcal Z })\le 1$ and 
$\#(\Z\cap B_{S,\mathcal Z}\setminus \mathcal M_{S,\mathcal Z})<\infty$, then
for a generic $(A,B)\in(\mathbf{W}^\vee)^2$, we have
$$\mdeg(\Sigma_{S}(\mathcal Z))=d(d-1)(10d-9) -\sum_{m\in \B}i_m(\mathcal Z,\mathcal P
     _{A,B}),$$
where $\mdeg(\Sigma_{S}(\mathcal Z))$ is the degree with multiplicity of 
$(\Sigma_{S}(\mathcal Z))$
($\mdeg(\Sigma_{S}(\mathcal Z))=\delta \deg(\Sigma_{S}(\mathcal Z))$, 
see Proposition \ref{blowingup} for the property satisfied by $\delta$), 
where $d$ is the degree of $\mathcal Z$ 
and where $i_m(\mathcal Z,\mathcal P
     _{A,B})$ denotes the intersection number of $\Z$ with $\mathcal P_{A,B}$ at point $m$.
\end{thm}
Let us notice, that in this formula, we can replace $i_m(\mathcal Z,\mathcal P
     _{A,B})$ by $i_m(\mathcal Z,V(K_1,K_2))$, with the notations of Proposition \ref{K1K2K3}.
Indeed, we can take $A$ and $B$ such that $\# V(A,A\circ\sigma,K_3)<\infty$ (see the proof of
Proposition \ref{degreepolar}) and use $\mathcal P_{a,b}\setminus V(A,A\circ\sigma)=V(K_1,K_2)\setminus
 V(A,A\circ\sigma)$).

Let us recall that the case when $\dim  (\mathcal Z\cap \mathcal M_{S,\mathcal Z })> 1$
has been studied in Lemma \ref{dim>1}.

Observe that, for $m\in \Z\cap B_{S,\mathcal Z}\setminus \mathcal M_{S,\mathcal Z}$, the reflected line $\mathcal R_m$
is well defined and contained in $\Sigma_S(\mathcal Z)$. Therefore, in the degenerate case when 
$\dim(\Z\cap B_{S,\mathcal Z}\setminus \mathcal M_{S,\mathcal Z })\ge 1$, the surface constituted by the reflected lines 
$\mathcal R_m$ for $m\in \Z\cap B_{S,\mathcal Z}$ is contained in $\Sigma_S(\mathcal Z)$. 

\begin{proof}[Proof of Theorem \ref{lemmefondamental}]
Recall that for a generic $(A,B)$ in $(\mathbf{W}^\vee)^2$, we have $\mbox{deg}(\mathcal P_{A,B})=(d-1)(10d-9)$.
We assume first that $\dim\Sigma_{S}(\mathcal Z)<2$ (i.e. $\delta=\infty$)
and that $\# \mathcal B<\infty$,
Taking $(A,B)$ such that
$\mbox{deg}(\mathcal P_{A,B})=(d-1)(10d-9)$ and $\mathcal D_{A,B}\cap \Sigma_{S}(\mathcal Z)=\emptyset$,
we have $\mathcal P_{A,B}\cap\mathcal Z=\B$ and so
\begin{eqnarray*}
d(d-1)(10d-9)&=&\mbox{deg}(\mathcal Z)\, \mbox{deg}(\mathcal P_{A,B})
= \sum_{m\in\mathcal Z\cap \mathcal P_{A,B}} i_m(\mathcal Z,\mathcal P_{A,B})\\
&=& \sum_{m\in \B} i_m(\mathcal Z,\mathcal P_{A,B})
= \sum_{m\in \B} i_m(\mathcal Z,\mathcal P_{A,B}).
\end{eqnarray*}

Assume now that $\dim\Sigma_{S}(\mathcal Z)=2$ (i.e. that $\delta$ is finite), that
$\dim  (\mathcal Z\cap \mathcal M_{S,\mathcal Z })\le 1$, that
$\#(\Z\cap B_{S,\mathcal Z}\setminus\mathcal M_{S,\mathcal Z})<\infty$
and $\# \mathcal B<\infty$.
We consider $(A,B)\in(\mathbf W^\vee)^2$ such that:
\begin{itemize}
\item[(a)] $\mathcal D_{A,B}$ is a line containing no reflected line $\mathcal R_m=(m\, \sigma(m))$ ($m\in\Z$), 
\item[(b)] $\mbox{deg}(\mathcal P_{A,B})=(d-1)(10d-9)$ (this is generic due to Proposition \ref{degreepolar}),
\item[(c)] the points $P\in \mathcal D_{A,B}\cap \Sigma_S(\mathcal Z)$ are such that $\#[\pi_1(\Phi_{|\hat Z}^{-1}(\{P\}))\setminus\mathcal B]=\delta$
(this is generic due to Proposition \ref{blowingup}),
\item[(d)] For any $P\in \mathcal D_{A,B}\cap \Sigma_{S}(\mathcal Z)$, we have $i_P(\Sigma_{S}(\mathcal Z),
    \mathcal D_{A,B})=1$ (this is true for a generic $(A,B)$ since
$\Sigma_{S}(\mathcal Z)$ is a surface),
\item[(e)] the line $\mathcal D_{A,B}$ intersects no reflected line $\mathcal R_m$ with $m\in B_{S,\mathcal Z}$ (this is generic since $\#(\mathcal Z
\cap B_{S,\mathcal Z}\setminus \mathcal M_{S,\mathcal Z })<\infty$ and $\dim(\mathcal Z\cap\mathcal M_{S,\mathcal Z})\le 1$),
\item[(f)] for any $m\in (\mathcal P_{A,B}\cap \mathcal Z)\setminus \B$, we have 
$i_m(\mathcal Z,\mathcal P_{A,B})=1$ (this is explained at the end of this proof),
\item[(g)] $\mathcal D_{A,B}$ does not intersect $\Phi(\pi^{-1}(\mathcal Z\cap V(\beta^2-4\alpha\gamma)))$ if 
$\dim(\mathcal Z\cap V(\beta^2-4\alpha\gamma))=1$.
\end{itemize}
Due to (b), we have
\begin{eqnarray*}
d(d-1)(10d-9)&=&\mbox{deg}(\mathcal Z)\, \mbox{deg}(\mathcal P_{A,B})
= \sum_{m\in\mathcal Z\cap \mathcal P_{A,B}} i_m(\mathcal Z,\mathcal P_{A,B})\\
&=& \sum_{m\in \B} i_m(\mathcal Z,\mathcal P_{A,B})+
      \sum_{m\in(\mathcal Z\cap \mathcal P_{A,B})\setminus \B} i_m(\mathcal Z,\mathcal P_{A,B}).
\end{eqnarray*}
Now, we have

$\displaystyle\sum_{m\in(\mathcal Z\cap \mathcal P_{A,B})\setminus \B} i_m(\mathcal Z,\mathcal P_{A,B})
=\#((\mathcal Z\cap \mathcal P_{A,B})\setminus \B)\ \ \ \mbox{due to}\
 (f)$
\begin{eqnarray*}
&=&\#[\pi_1(\Phi_{|\hat Z}^{-1}(\mathcal D_{A,B}))\setminus\mathcal B]\ \ \mbox{due to}\ \ Definition \ \ref{defirefpolar}\\
&=&\delta \#(\Sigma_{S}(\mathcal Z)\cap  \mathcal D_{A,B})
\ \ \ 
\mbox{due to}\ (c)\\
&=&\delta\sum_Pi_P(\Sigma_{S}(\mathcal Z),  \mathcal D_{A,B})=\delta\, \mbox{deg}
(\Sigma_{S}(\mathcal Z))\ \ \ 
\mbox{due to}\ (d).
\end{eqnarray*}
Let us now explain why (f) is true for a generic $(A,B)\in(\mathbf W^\vee)^2$. 
Let $m\in(\mathcal P_{A,B}\cap \mathcal Z)\setminus\mathcal B$.
Due to (e), $m\in \pi_1(\Phi_{|\hat Z}^{-1}(\mathcal D_{A,B}))\setminus
B_{S,\mathcal Z}$.
We consider the cone hypersurface
$\mathcal K_\mathcal Z$ of $\mathbf{W}$ associated to $\mathcal Z$.
Since $m\in\mathcal Z\setminus B_{S,\Z}$, there exist two maps $\psi^\pm:U\rightarrow\mathbb P^3$
defined on a neighbourhood $U$ of $m$ in $\mathbb P^3$ such that, for any $m'\in U$, 
$\Phi(\pi_1^{-1}(\{m'\}))=\{\psi^-(m'),\psi^+(m')\}$. Let $\varepsilon\in\{+,-\}$
be such that $\Phi(\pi_1^{-1}(\{m\}))\cap \mathcal D_{A,B}=\{\psi^{\varepsilon}(m)\}$ ($\psi^\varepsilon$ is unique for a generic $m\in \Z$ according to (a) and to (g))
and the tangent space to $\mathcal P_{A,B}$ at $m$ is given by
$V(A\circ D\psi^\varepsilon(\mathbf{m}), B\circ D\psi^\varepsilon
   (\mathbf{m}))$,
where $D\psi^\pm(\mathbf{m})$ are the jacobian matrices of $\psi^\pm$ taken at $\mathbf{m}$.
Now, with these notations, for a generic $m$ in $\mathcal Z$, $D\psi^{\varepsilon'}(\mathbf{m})$ is invertible
if $\dim \psi^{\varepsilon'}(\mathcal Z)=2$ (if $\dim \psi^{\varepsilon'}(\mathcal Z)<2$, take $(A,B)$ such that $\mathcal D_{A,B}\cap 
\psi^{\varepsilon'}(\mathcal Z)=\emptyset$).
This combined with (d) gives the result.
\end{proof}
\section{Proof of Theorem \ref{thmGENE}}\label{proofGENE}
More precisely we prove the following (recall that $\#V(F,\Delta_{\mathbf S}F,Q(\nabla F))$ is the number of isotropic tangent planes to $\mathcal Z$ passing through $S$).
\begin{thm}\label{generique}
Let $\mathcal Z\subset\mathbb P^3$ be an irreducible smooth surface and $S\in\mathbb P^3\setminus(\mathcal Z\cup\mathcal H^\infty)$
be such that 
$\mathcal B=V(F,\Delta_{\mathbf S}F,Q(\nabla F))$ (see Corollary \ref{B}) and such that this set contains $d(d-1)(2d-2)$ points.
We assume moreover that $\mathcal B\cap V(H_FN_{\mathbf S})=\emptyset$.
Then 
$\mdeg (\Sigma_S(\mathcal Z))=d(d-1)(8d-7)$.
\end{thm}
\begin{proof}
Without any loss of generality, we assume that $S[0:0:0:1]$.
Due to Theorem \ref{lemmefondamental}, we have
$\mdeg(\Sigma_{S}(\mathcal Z))=d(d-1)(10d-9) -\sum_{P\in \B}i_P(\mathcal Z,V(K_1,K_2))$
for generic $A,B\in\mathbf W^\vee$.
Since we know that $\mathcal B=V(F,\Delta_{\mathbf S}F,Q(\nabla F))$ contains $d(d-1)(2d-2)$ points, we just have to prove that
$i_P(\mathcal Z,\mathcal P_{A,B})=1$ for any $P\in\mathcal B$
(for generic $A,B\in\mathbf W^\vee$).

Let such a point $P$. 
Assume that $A(P)B(P)\ne 0$ (this is true for generic $A,B\in\mathbf W^\vee$).
Since $F$ is smooth, either $F_x(\mathbf P)F_y(\mathbf P)\ne 0$, or $F_x(\mathbf P)F_z(\mathbf P)\ne 0$
or $F_y(\mathbf P)F_z(\mathbf P)\ne 0$. Assume for example that $F_x(\mathbf P)F_y(\mathbf P)\ne 0$,
then there exists a local parametrization $h$ of $\mathcal Z$
defined on an open  neighbourhood of $(0,0)$ in $\mathbb C^2$
such that $h(0,0)=P$ and $[\partial h(u,v)/\partial u](0,0)\ne 0$.
Since $\# V(F,\Delta_{\mathbf S}F,Q(\nabla F))=d(d-1)(2d-2)$,
we know that $i_P(V(F,\Delta_{\mathbf S}F,Q(\nabla F)))=1$.
Hence we assume that
$\val_{u,v}A(\boldsymbol\sigma(h(u,v)))=
              \val_{u,v}B(\boldsymbol\sigma(h(u,v)))=1$.
(this is true for generic  $A,B\in\mathbf W^\vee$ due to the formula of $\boldsymbol \sigma$). Moreover $\Delta_{\mathbf S}F(h(u,v))$, $\alpha(h(u,v))$, $\beta(h(u,v))$ and $\gamma(h(u,v))$ have valuation 1 in $(u,v)$.
Recall that $K_1$ and $K_2$ are given by
$ K_1(\mathbf m)=A(\boldsymbol \sigma(\mathbf m))B(\mathbf m)
     -B(\boldsymbol \sigma(\mathbf m))A(\mathbf m)$ and 
$K_2(\mathbf m)=\alpha(\mathbf m)
    (A(\boldsymbol\sigma(\mathbf m)))^2
    -\beta(\mathbf m)A(\boldsymbol\sigma(\mathbf m))A(\mathbf m)+
     \gamma(\mathbf m)(A(\mathbf m))^2$.
On the one hand $K_2(h(u,v))$ has valuation 1  and its term of degree 1 
is the term of degree 1 of 

$\gamma(h(u,v)) 
(A(\mathbf P))^2=-4(A(\boldsymbol\sigma(\mathbf P)))^2\Delta_{\mathbf S}F(h(u,v))N_{\mathbf S}(P)H_F(\mathbf P)/(d-1)^2$.
On the other hand
$$K_1(h(u,v))=\Delta_{\mathbf S}F(h(u,v))[\cdots]+
        Q(\nabla F(h(u,v)))[A(\mathbf S)B(h(u,v))-B(\mathbf S)A(h(u,v))].$$
Hence, for generic $A,B\in\mathbf W^\vee$, $A(\mathbf S)B(\mathbf P)\ne B(\mathbf S)A(\mathbf P)$ and so $i_P(V(F,K_1,K_2))=1$. 
\end{proof}
\section{Degree of caustics of a paraboloid}\label{sec:paraboloid}
This section is devoted to the proof of Proposition \ref{prop:paraboloid}.
We consider again the case when $\mathcal Z$ is the 
paraboloid $V(F)$ with 
$F=(x^2+y^2-2zt)/2$. We recall that the base points have been
studied in Proposition \ref{baseparaboloide} and that we have written
$F_1[0:0:1:0]$ and $F_2[0:0:1:2]$ for the two focal points of this paraboloid. Let $S\in\mathbb P^3\setminus\{F_1,F_2\}$.
We have
$$F=(x^2+y^2-2zt)/2,\quad \Delta_{\mathbf S}F=x_0x+y_0y-z_0t-t_0z,$$
$$\boldsymbol\sigma=(x^2+y^2+t^2)\cdot\mathbf S-2\Delta_{\mathbf S}F\cdot \left(\begin{array}{c}x\\y\\-t\\0\end{array}\right), $$
$$ K_1(\mathbf m)=A(\boldsymbol \sigma(\mathbf m))B(\mathbf m)
     -B(\boldsymbol \sigma(\mathbf m))A(\mathbf m),$$
$$K_2(\mathbf m)=\alpha(\mathbf m)
    (A(\boldsymbol\sigma(\mathbf m)))^2
    -\beta(\mathbf m)A(\boldsymbol\sigma(\mathbf m))A(\mathbf m)+
     \gamma(\mathbf m)(A(\mathbf m))^2$$
$$\mbox{with}\quad
\alpha=\Delta_{\mathbf S}F,\quad \gamma=-4\Delta_{\mathbf S}FN_{\mathbf S}, $$
$$\beta=-2(x_0^2+y_0^2-2z_0t_0)(x^2+y^2+t^2)+
     4\Delta_{\mathbf S}F(t_0z+z_0t+t_0t),$$
$$\mbox{and}\quad N_{\mathbf S}=(x_0t-xt_0)^2+(y_0t-yt_0)^2+(z_0t-zt_0)^2 .$$
It will be useful to observe that $\mathcal Z$ is invariant by composition by $[x:y:z:t]\mapsto[\bar x:\bar y:\bar z:\bar t]$ and by
$[x:y:z:t]\mapsto[ax+by:-bx+ay:cz:t/c]$ with $a^2+b^2=c^2=1$.
\begin{itemize}
\item 
If $x_0=y_0=0$, then, due to Theorem \ref{thmrevolution}, 
the degree of $\Sigma_S(\mathcal Z)$ corresponds
to the degree of a caustic of the parabola.
Hence, we have $\mdeg\Sigma_S(\mathcal Z)=4$ if $S[0:0:0:1]$ and
$\mdeg\Sigma_S(\mathcal Z)=6$ elsewhere (see \cite{fredsoaz1}).
\end{itemize}
We assume now that $x_0\ne 0$.
\begin{itemize}
\item {\bf [Generic case]} If $S\not\in(\mathcal Z\cup\mathcal H^\infty)$, if $x_0^2+y_0^2\ne 0$ and if $x_0^2+y_0^2+(z_0-(t_0/2))^2\ne 0$, then Theorem \ref{generique} applies and $\mdeg \Sigma_S (\mathcal Z)=18$.
Indeed $\mathcal B=V(F,\Delta_{\mathbf S}F,Q(\nabla F))=\{C,D,m_1,m_{-1}\}$ with $C[t_0:it_0:x_0+iy_0:0]$, 
$C[t_0:-it_0:x_0-iy_0:0]$ and $m_{\varepsilon}[x_\varepsilon;y_\varepsilon:-1/2:1]$ with
$x_\varepsilon$ and $y_\varepsilon$ defined in Proposition
\ref{baseparaboloide}. Moreover  we have $N_{\mathbf S}(\mathbf C)=(x_0+iy_0)^2t_0^2\ne 0$, $N_{\mathbf S}(\mathbf D)=(x_0-iy_0)^2t_0^2\ne 0$ and $N_{\mathbf S}(\mathbf m_\varepsilon)=(x_0^2+y_0^2)+(z_0-\frac{t_0}2)^2\ne 0$.
\end{itemize}
In order to apply our Theorem \ref{lemmefondamental}, we will have to verify its assumptions
on $\mathcal B_{S,\mathcal Z}$ and $\mathcal M_{S,\mathcal Z}$.
This is the aim of the next proposition.
\begin{prop}
Let $S\in\mathbb P^3\setminus\{F_1,F_2\}$ be such that $\#
\mathcal B<\infty$. Then $\#(\mathcal Z\cap \mathcal M_{S,\mathcal Z})<\infty$. Moreover if $\hess F(\mathbf{S},\mathbf{S})\ne 0$ (i.e. if $S\not\in\mathcal Z$), then $\#(\mathcal Z\cap
\mathcal B_{S,\mathcal Z})<\infty$ and so Theorem \ref{lemmefondamental} applies.
\end{prop}
\begin{proof}
Let us prove that
$\#(\mathcal Z\cap \mathcal M_{S,\mathcal Z})<\infty$. Since $\#
\mathcal B<\infty$, we already know that $\# V(F,\Delta_{\mathbf{S}}F,Q(\nabla F))<\infty$. Moreover we have
$V(F,F_x,F_y,F_z)=\{[0:0:1:0]\}$ and $\mathcal W=
\{[1:\pm i:0:0]\}$ (indeed for $m\in\mathcal W$, we have
$(x,y)\ne 0$, so $z=t$  and $-z^2=x^2+y^2=2zt=2z^2$).

If $\hess F(\mathbf{S},\mathbf{S})\ne 0$, then  $\mathcal Z\cap\mathcal B_{S,\mathcal Z}=V(F,\Delta_{\mathbf{S}}F,Q(\nabla F))\subseteq \mathcal B$ which is finite.
\end{proof}

\begin{itemize}
\item 
Let $S\not\in(\mathcal Z\cup\mathcal H^\infty)$ such that $x_0^2+y_0^2\ne 0$ and $x_0^2+y_0^2+(z_0-(t_0/2))^2= 0$.
Without loss of generality we assume that $x_0=0$ and $y_0=1$
and $z_0-(t_0/2)=i$. The fact that $S\not\in\mathcal Z$ implies
that $t_0\ne -i$.
We have $\mathcal B=V(F,\Delta_{\mathbf S}F,Q(\nabla F))=
\{C,D,E\}$ with $C[t_0:it_0:i:0]$, $D[t_0:-it_0:-i:0]$ and $E[0:i:-\frac 12:1]$. 

Around $E$, we parametrize $\mathcal Z$ by $h(x,y)=(x,i+y,\frac{x^2+(i+y)^2}{2},1)$. We have
$\Delta_{\mathbf S}F\circ h(x,y)=y(1-it_0)-t_0\frac{x^2+y^2}2$ and
$Q(\nabla F)\circ h(x,y)=2iy+x^2+y^2$.
Hence $i_E(\mathcal Z,V(\Delta_{\mathbf S}F,Q(\nabla F)))=2$
and so $i_C(\mathcal Z,V(\Delta_{\mathbf S}F,Q(\nabla F)))=1$
and $i_D(\mathcal Z,V(\Delta_{\mathbf S}F,Q(\nabla F)))=1$.
Since moreover $N_{\mathbf S}(\mathbf C)=(x_0+iy_0)^2t_0^2\ne 0$ and $N_{\mathbf S}(\mathbf D)=(x_0-iy_0)^2t_0^2\ne 0$, due to the proof of Theorem \ref{generique}, we have
$i_C(\mathcal Z,V(\Delta_{\mathbf S}F,Q(\nabla F)))=i_D(\mathcal Z,V(\Delta_{\mathbf S}F,Q(\nabla F)))=1$.
Observe that for a generic $A\in\mathbf W^\vee$, $A\circ\boldsymbol{\sigma}\circ h$ has valuation 1 with dominating term
$2[a_2y(i-1+it_0)+a_3y(1-t_0i)+ia_4yt_0]$.
Hence $\val_{x,y}K_1\circ h(x,y)=1$ and its dominating term is proportional to $y$.
Using the fact that $z_0+\frac {t_0}2=i+t_0=i(1-it_0)$, we have
$N_{\mathbf S}\circ h(x,y)=-it_0x^2-it_0(i+t_0)y^2+...$, $\val_{x,y}\alpha\circ h(x,y)=1$, $\val_{x,y}\gamma\circ h(x,y)=3$.
Moreover $x_0^2+y_0^2-2z_0t_0=(1-it_0)^2$ so
$\beta\circ h(x,y)=-2x^2(1-it_0)-2y^2(1-it_0)^2+... $.
Therefore $\val_{x,y}K_2\circ h(x,y)=3$ and we conclude that
$i_{E}(\mathcal Z,V(K_1,K_2))=3$ and that
$\mdeg \Sigma_S(\mathcal Z)=22-1-1-3=17$.
\item 
If $S\not\in(\mathcal Z\cup\mathcal H^\infty)$, if $x_0^2+y_0^2=0$
($x_0\ne 0$) and $z_0\ne t_0/2$. We assume without loss of generality that $x_0=1$ and $y_0=i$. 
$\mathcal B=V(F,\Delta_{\mathbf S}F,Q(\nabla F))=\{C,D,E\}$
with $C[1:i:0:0]$, $D[t_0:-it_0:2:0]$ and $E[\frac{z_0-(t_0/2)}{2x_0}:\frac{z_0-(t_0/2)}{2y_0}:-\frac 12:1]$.

We use the parametrization $h(z,t)=(1,i\sqrt{1-2zt},z,t)$ at a neighbourhood of $\mathcal Z$ around $C$. We have
$\Delta_{\mathbf S}F(h(z,t))=-z_0t-t_0z+(1-\sqrt{1-2zt})$ and $Q(\nabla F)(h(z,t))=(2z+t)t$ and so $i_C(\mathcal Z,V(\Delta_{\mathbf S}F,Q(\nabla F)))=2$ and so the intersection
numbers of $\mathcal Z$ with $V(\Delta_{\mathbf S}F,Q(\nabla F))$
is equal to 1 at $D$ and $E$. Since $N_{\mathbf S}(\mathbf D)=
    (x_0-iy_0)^2t_0^2\ne 0$ and $N_{\mathbf S}(\mathbf E)=
    (z_0-\frac{t_0}2)^2\ne 0$.
Hence, due to the Proof of Theorem \ref{generique}, we have
$ i_D(\mathcal Z,V(\Delta_{\mathbf S}F,Q(\nabla F)))=i_E(\mathcal Z,V(\Delta_{\mathbf S}F,Q(\nabla F)))=1$.
It remains to estimate $i_C(\mathcal Z,V(\Delta_{\mathbf S}F,Q(\nabla F)))$. We have
$$\boldsymbol\sigma\circ h(z,t)=
\left(\begin{array}{c}t^2+2(z_0t+t_0z+\sqrt{1-2zt}-1-zt)\\2(z_0t+t_0z+\sqrt{1-2zt}-1)i\sqrt{1-2zt}+2izt+it^2\\
2(z_0+t_0)tz-z_0t^2-2(\sqrt{1-2zt}-1)t\\(2z+t)t_0t\end{array}\right).$$
Hence, for generic $A,B\in\mathbf W^\vee$, $K_1\circ h(z,t)$
has valuation 2 with dominating terms
$$[z_0(2z+t)^2+(t_0-2z_0)2z(z+t)]((a_1+ia_2)b_3-(b_1+ib_2)a_3)+ $$
$$+t^2(t_0-2z_0)[(b_1+ib_2)a_4-b_4(a_1+ia_2)]. $$
Moreover
$\val_{z,t} \alpha\circ h(z,t)=1$, $\val_{z,t}\beta\circ h(z,t)=2$ and  $\val_{z,t}\gamma\circ h(z,t)=2$. Hence $K_2\circ h(z,t)=
(4(z_0t-t_0z)^2-8(z_0-2t_0)zt)(z_0t+t_0z) (a_1+ia_2)^2 +...$
for a generic $A\in\mathbf W^\vee$, 
so $i_C(\mathcal Z,V(K_1,K_2))=6$
and $\mdeg \Sigma_S(\mathcal Z)=22-1-1-6=14$.
\item If $S\not\in(\mathcal Z\cup\mathcal H^\infty)$, if $x_0^2+y_0^2=0$
($x_0\ne 0$) and $z_0= t_0/2$. We assume without loss of generality that $x_0=1$ and $y_0=i$. 
$\mathcal B=V(F,\Delta_{\mathbf S}F,Q(\nabla F))=\{C,D\}$
with $C[1:i:0:0]$, $D[t_0:-it_0:2:0]$.
We use again the parametrization $h(z,t)=(1,i\sqrt{1-2zt},z,t)$ at a neighbourhood of $\mathcal Z$ around $C$. We observe that
$i_C(\mathcal Z,V(\Delta_{\mathbf S}F,Q(\nabla F)))=3$
and so $i_D(\mathcal Z,V(\Delta_{\mathbf S}F,Q(\nabla F)))=1$.
Moreover $N_{\mathbf S}(\mathbf D)=
    (x_0-iy_0)^2t_0^2\ne 0$. Hence, due to the proof of Theorem
\ref{generique}, we have
$ i_D(\mathcal Z,V(K_1,K_2))=1.$
Moreover,  we prove that
$ i_C(\mathcal Z,V(K_1,K_2))=9$ (probranches of $K_1\circ h$
and $K_2\circ h$ have intersection number 3/2)
and so $\mdeg \Sigma_S(\mathcal Z)=22-1-9=12$.
\end{itemize}

\begin{prop}
Let $S\in\mathcal Z\setminus (\mathcal D\cup \mathcal H^\infty)$. Then
$\mdeg \Sigma_S(\mathcal Z)=12$ if $x_0^2+y_0^2+t_0^2=0$,
$\mdeg \Sigma_S(\mathcal Z)=14$ if $x_0^2+y_0^2=0$ and
$\mdeg \Sigma_S(\mathcal Z)=16$ otherwise.
\end{prop}
\begin{proof}
Observe that $\Delta_{\mathbf S}F$ divides $\alpha$, $\beta$
and $\gamma$. In this case we define $\Sigma_S(\mathcal Z)$
by replacing $Q_{\mathbf S,F}$ by $\tilde Q_{\mathbf S,F}:=Q/\Delta_{\mathbf S}F$,
we define analogously $\tilde\alpha:=1$, $\tilde\beta:=4(t_0z+z_0t+t_0t)$ and $\tilde\gamma:=-4N_{\mathbf S}$. Following
our argument above, we define $\tilde {\mathcal B}_{S,\mathcal Z}:=V(\tilde\alpha,\tilde\beta,\tilde\gamma)=\emptyset$ and
the new reflected polar curve $\tilde {\mathcal P}_{A,B}=V(K_1,\tilde K_2,\tilde K_3)$ by using $\tilde Q_{\mathbf S,F}$ instead of $Q_{\mathbf S,F}$. Following the proof of 
Proposition \ref{degreepolar}, we obtain that $\deg \mathcal P_{A,B}=\deg K_1\deg \tilde K_2-2\deg A\deg (A\circ\sigma)=8$
and the corresponding set of base points $\tilde{\mathcal B}$ is contained in $(\{S\}\cap\mathcal Z)\cup V(F,\Delta_SF,Q(\nabla F),N_{\mathbf S}))$ (recall that $\mathcal B=(\{S\}\cap\mathcal Z)\cup V(F,\Delta_SF,Q(\nabla F)))$. 
\begin{itemize}
\item 
If $S\in \mathcal Z\setminus\mathcal H^\infty$,
$x_0^2+y_0^2\ne 0$ and $x_0^2+y_0^2+t_0^2\ne 0$
(so $z_0+\frac{t_0}2\ne 0$), then
$V(F,\Delta_{\mathbf S}F,Q(\nabla F))=\{C,D,m_1,m_{-1}\}$ 
with $C[t_0:-it_0:x_0-iy_0:0]$, $D[t_0:it_0:x_0+iy_0:0]$
and $m_\varepsilon[x_\varepsilon:y_{\varepsilon}:-\frac 12:1]$
with $x_\varepsilon:=\frac{x_0}{2z_0t_0}(z_0-\frac{t_0}2)+i\varepsilon \frac{y_0}{2t_0z_0}(z_0+\frac {t_0}2)$ and $y_\varepsilon:=\frac{y_0}{2z_0t_0}(z_0-\frac{t_0}2)-i\varepsilon \frac{x_0}{2t_0z_0}(z_0+\frac {t_0}2)$.
Hence the intersection number of $\mathcal Z$ with
$V(\Delta_{\mathbf S}F,Q(\nabla F))$ is 1 at these four points.
We observe that $N_{\mathbf S}(\mathbf C)=t_0^2(x_0+iy_0)^2\ne 
0$, $N_{\mathbf S}(\mathbf D)=t_0^2(x_0-iy_0)^2\ne 0$
and $N_{\mathbf S}(\mathbf m_\varepsilon)=(z_0+\frac {t_0}2)\ne 0$. Hence
$i_C(V(F,K_1,\tilde K_2))=i_D(V(F,K_1,\tilde K_2))=i_{m_\varepsilon}(V(F, K_1,\tilde K_2))=0$.

It remains to compute $i_S(V(F, K_1,\tilde K_2))$. 
We have $\sigma(S)=(x_0^2+y_0^2+t_0^2)\cdot S$ and $\tilde K_2(S)=
-3(x_0^2+y_0^2+t_0^2)^2(A(S))^2\ne 0$.
We conclude that
$\mdeg\Sigma_S(\mathcal Z)=2\times 8=16$. 

\item If $S\in \mathcal Z\setminus\mathcal H^\infty$ and $x_0^2+y_0^2+t_0^2= 0$, then $z_0=-t_0/2$
and so
$\mathcal B= V(F,\Delta_{\mathbf S}F,Q(\nabla F))=\{S,C,D\}$
with $C[t_0:it_0:x_0+iy_0:0]$ and $D[t_0:-it_0:x_0-iy_0:0]$.
Assume $t_0=1$.

Observe that $N_{\mathbf S}(\mathbf C)\ne 0$ and so
that $i_C(V(F, K_1,\tilde K_2))=0$.
Analogously we have $i_D(V(F, K_1,\tilde K_2))=0$.
Using the parametrization $h(x,y)=(x_0+x,y_0+y,\frac{(x_0+x)^2+(y_0+y)^2}2,1)$ of $\mathcal Z$
around $S$ and the fact that
$\tilde\alpha=1$, $\tilde\beta\circ h(x,y)=4(xx_0+yy_0+\frac{x^2+y^2}2)$,  that $\tilde\gamma\circ h(x,y)=-4[x^2+y^2+(xx_0+yy_0+(x^2+y^2)/2)^2]$.
Moreover 
$$\boldsymbol\sigma\circ h(x,y)=
(2xx_0+2yy_0+x^2+y^2)\cdot \mathbf S+\frac{x^2+y^2}2\cdot\left(
     \begin{array}{c}x_0+x\\y_0+y\\-1\\0\end{array}\right) .$$
Hence $\tilde K_1\circ h(x,y)=\frac{x^2+y^2}2[(a_1x_0+a_2y_0-a_3)B(S)
         -(b_1x_0+b_2y_0-b_3)A(S)]$
and $\tilde K_2:=16(xx_0+yy_0)^2(A(S))^2$. 
So $i_S(V(F,\tilde K_1,\tilde K_2))=4$ and 
$\mdeg\Sigma_S(\mathcal Z)=2\times 8-4=12$.
\item If $S\in \mathcal Z\setminus\mathcal H^\infty$ and  $x_0^2+y_0^2= 0$, then 
we assume without loss of generality that $x_0=1$ and $y_0=i$ (so
$z_0=0$). We have $\mathcal B=\{
S,C,D,E\}$ with $C[1:i:0:0]$, $D[t_0:-it_0:2:0]$
and $E[-\frac{t_0}{4x_0}:-\frac{t_0}{4y_0}:-\frac 12:1]$.
We have $N_{\mathbf S}(\mathbf D)=4t_0^2$ and
$N_{\mathbf S}(\mathbf E)=5t_0^2/4$ and so
$i_D(V(F, K_1,\tilde K_2))=i_E(V(F, K_1,\tilde K_2))=0$.
Observe that $\Delta_{\mathbf S}F(\mathbf S)=0$, $N_{\mathbf S}(\mathbf S)=0$, 
$\boldsymbol\sigma(\mathbf S)=t_0^2\cdot\mathbf S$ and
$\tilde\beta(\mathbf S)=4t_0^2\ne 0$. So $\tilde K_2(\mathbf S)\ne 0$ and $i_S(\mathcal Z,V(K_1,\tilde K_2))=0$.

Around $C$, we parametrize $\mathcal Z$ by
$h(z,t)=(1,i\sqrt{1-2zt},z,t)$.
We have $\Delta_{\mathbf S}F\circ h(z,t)=1-\sqrt{1-2zt}-t_0z$ and
$$\boldsymbol\sigma\circ h(z,t)= 
  \left(\begin{array}{c}2t_0z+2zt+t^2+2(\sqrt{1-2zt}-1)\\
2it_0z\sqrt{1-2zt}+2izt+it^2+2i(\sqrt{1-2zt}-1)\sqrt{1-2zt}
\\2t_0tz-2t(\sqrt{1-2zt}-1)\\(2z+t)t_0t\end{array}\right).$$
Hence $A\circ \boldsymbol\sigma\circ h(z,t)$ has valuation 1.
Moreover $\tilde\alpha=1$, $\tilde\beta\circ h(z,t)=4(t_0z+t_0t)$
and $\tilde\gamma\circ h(z,t)=8t_0t+...$. Hence, for a generic
$A\in\mathbf W^\vee$, $\tilde K_2\circ h(z,t)=8t_0tA(C)^2+...$.
Moreover we have
$$ K_1\circ h(z,t)=2t_0\{(a_3tz+a_4t^2-a_3z^2)
(b_1+ib_2)-(b_3tz+b_4t^2-b_3z^2)(a_1+ia_2)\}+...$$
Hence $i_C(\mathcal Z,V( K_1,\tilde K_2))=2$ and so $\mdeg \Sigma_S(\mathcal Z)=16-2=14$.
\end{itemize}
\end{proof}
We assume now that $S$ is at infinity. In this case $N_{\mathbf S}=(x_0^2+y_0^2+z_0^2)t^2$.
\begin{itemize}
\item 
If $S\in\mathcal H^\infty\setminus(\mathcal Z\cup \mathcal C_\infty)$, then $x_0^2+y_0^2\ne 0$ and
$\mathcal B=\{F_1,m_1,m_{-1}\}$ with
$m_\varepsilon[x_0z_0+i\varepsilon y_0\sqrt{x_0^2+y_0^2+z_0^2}:y_0z_0-i\varepsilon x_0\sqrt{x_0^2+y_0^2+z_0^2}:-
\frac{x_0^2+y_0^2}2:x_0^2+y_0^2]$. 
Then we will prove that the intersection number of $V( F,K_1,K_2)$ is 8
at $F_1$ and 1 at the two other base points and so
 $\mdeg\Sigma_{S}(\mathcal Z)=22-8-1-1=12$.

Due to the proof of Theorem \ref{generique}, since $N_{\mathbf S}(m_\varepsilon)\ne 0$, to prove that 
$i_{m_\varepsilon}(\mathcal Z,V(K_1,K_2))=1$ it is enough to prove
that $i_{m_\varepsilon}(V(F,\Delta_{\mathbf S}F,Q(\nabla F))=1$. 
To see this, we use the parametrization $h(x,y)=(x_1+x,y_1+y,((x_1+x)^2+(y_1+y)^2)/2,1)$ of $\mathcal Z$
 around $m_\varepsilon=[x_1:y_1:-1/2:1]$. The terms of valuation 1
of $\Delta_{\mathbf S}F\circ h$ and $Q(\nabla F)\circ h$ are respectively $x_0x+y_0y$
and $2(xx_1+yy_1)$ which are not proportional since $x_0^2+y_0^2\ne 0$.

For $F_1$, we use the parametrization $h(x,y)=(x,y,1,(x^2+y^2)/2)$
of $\mathcal Z$ around $F_1$. We observe that, for generic $A,B\in\mathbf W^\vee$, $A\circ\boldsymbol\sigma\circ h$ and $B\circ\boldsymbol\sigma\circ h$ have valuation 2 with respective dominating terms:
$$\theta_A:=a_1[-x_0x^2+x_0y^2-2y_0xy]+a_2[-y_0y^2+y_0x^2-2x_0xy]+a_3(x^2+y^2)z_0,$$
$$\theta_B:=b_1[-x_0x^2+y_0y^2-2y_0xy]+b_2[-y_0y^2+x_0x^2-2x_0xy]+b_3(x^2+y^2)z_0.$$
Therefore the lowest degree terms of $K_1\circ h$ are given by
$$(b_3a_1-a_3b_1)[-x_0x^2+x_0y^2-2y_0xy]+(b_3a_2-a_3b_2)[-y_0y^2+y_0x^2-2x_0xy].$$
Moreover the valuations of $\alpha\circ h$,
$\beta\circ h$ and $\gamma\circ h$ are respectively 1, 2 and 5.
Hence $K_2\circ h$ has valuation 4 and its dominating term
is 
$a_32(x_0^2+y_0^2)(x^2+y^2)\theta_A$. Hence the curves of equations $K_1\circ h$ and $K_2\circ h$ are transverse and we conclude that $i_{F_1}(\mathcal Z,V(K_1,K_2))=8$.
\item 
If $S\in\mathcal C_\infty\setminus\mathcal Z$, then $x_0^2+y_0^2\ne 0$, $z_0\ne 0$ and
$\mathcal B=\{F_1,m_1\}$ with
$m_1[x_0z_0:y_0z_0:-
\frac{x_0^2+y_0^2}2:x_0^2+y_0^2]$. Then $N_{\mathbf S}=0$
in $\mathbb C[x,y,z,t]$ (so $\gamma=0$ in $\mathbb C[x,y,z,t]$ ). 
We prove that $i_{F_1}(\mathcal Z,V(K_1,K_2))=8$ as in the previous case. We compute $i_{m_1}(\mathcal Z,V(K_1,K_2))$. 
We assume without loss of generality that $x_0=0$, $y_0=1$
and $z_0=i$.
Around
$m_1[0:z_0:-1/2:1]$, we  parametrise $\mathcal Z$ by $h(x,y)=(x,z_0+y,((x_1+x)^2+(y_1+y)^2)/2,1)$.
We have $\alpha\circ h(x,y)=\Delta_{\mathbf S}F\circ h(x,y)=y$, $Q(\nabla F)\circ h(x,y)=x^2+2iy+y^2$ and
$$\boldsymbol\sigma\circ h(x,y)= (x^2+2iy+y^2)
   \left(\begin{array}{c}0\\1\\i\\0\end{array}\right)-2y
    \left(\begin{array}{c}x\\i+y\\-1\\0\end{array}\right)=\left(\begin{array}{c}-2xy\\x^2-y^2\\z_0(x^2+y^2)\\0\end{array}\right).$$ 
Hence, for generic $A,B\in\mathbf W^\vee$, 
$K_1\circ h(x,y)$ has valuation 2. Moreover
$\alpha(h(x,y))=y$, $\beta(h(x,y))=-2(x^2+y^2)$ and
$\gamma(h(x,y))=0$.
Hence $K_2\circ h(x,y)$ has valuation 4 and we have
$i_{m_1}(\mathcal Z,V(K_1,K_2))=8$, so
$\mdeg\Sigma_S(\mathcal Z)=22-8-8=6$.
\item 
The case when $S\in\mathcal Z\cap\mathcal H^\infty$
has been studied in Proposition \ref{cas1}.
\end{itemize}
\section{About a reflected bundle}\label{bundle}
Recall that $\mathcal O_{\mathcal Z}(-1)=\{(m,v)\in\mathcal Z\times \mathbf{W}\ :\ v\in m\}$.
Observe that the set $\mathbf{R}(-1)$ of $(m,v)$ in the trivial bundle $\mathcal Z\times \mathbf{W} $ such that
$v$ corresponds to a point of $\mathbb P^3$ on the reflected line $\mathcal R_m$ is:
$$\mathbf{R}(-1)=\mathcal O_{\mathcal Z}(-1)+\{(m,v)\in \mathcal Z\times \mathbf{W}\ :\ v\in \sigma(m)\}.$$
Observe that this sum is direct in the generic case (when $S\not\in\mathcal Z$ and when $\mathcal W=\emptyset$, see Proposition \ref{B1}).
But, contrarily to the normal bundle considered in \cite{Trifogli,CataneseTrifogli} to study the evolute, 
$\mathbf{R}(-1)$ does not define a bundle since its rank is not constant. Indeed, the dimension of
$Vect(\mathbf{m},\boldsymbol{\sigma}(\mathbf{m}))$ equals 2 in general but not at every point $m\in\mathcal Z$ (it is strictly less than
2 when $m$ is a base point of $\mathbf{\sigma}_{|\mathcal Z}$ 
and, as seen in Proposition \ref{baseR}, such points always exist).
\begin{appendix}
\section{Caustics of surfaces linked with caustics of curves}\label{decompose}
For the classes of examples studied in this section, 
caustics of surfaces are linked with of caustics of planar curves.
We start with some facts on caustics of planar curves.
\subsection{Caustic of a planar curve}\label{lien}
Let $S_0[x_0:y_0:t_0]\in\mathbb P^2$ and an irreducible algebraic curve $\mathcal C=V(G)\subset \mathbb P^2$ with $G\in\mathbb C[x,y,t]$ homogeneous of degree $d\ge 2$. 
We write $\Delta_{\mathbf S_0}G:=x_0G_x+y_0G_y+z_0G_z$,
$N_{\mathbf S_0}=(x_0t-xt_0)^2+(y_0t-yt_0)^2$,
$\hess G$ for the Hessian form of $G$ and $H_G$ for its determinant
and  $\sigma_{\mathbf S_0,G}=(G_x^2+G_y^2)\cdot\mathbf S-2\Delta_{\mathbf S_0}G\cdot(G_x,G_y,0)$.
\begin{defi}[\cite{fredsoaz1}]\label{thmcurve}
The {\bf caustic map} of $\mathcal C$ from $S_0$
is the rational map $\Phi_{{{S_0}},\mathcal C}:\mathbb P^2\rightarrow\mathbb P^2$  corresponding
to $\boldsymbol {\Phi}_{{\mathbf{S}}_0,G}:\mathbb C^3\rightarrow\mathbb C^3$ given by 
$\boldsymbol {\Phi}_{{\mathbf{S}}_0,G}=-\frac{2H_GN_{\mathbf {S}_0}}{(d-1)^{2}} \cdot \Idd
 +\Delta_{{\mathbf {S}}_0}G\cdot
\boldsymbol{\sigma}_{{\mathbf{S}_0},G}$.
The caustic by reflection $\Sigma_{S_0}(\mathcal C)$
is the Zariski closure of $\Phi_{{S}_0,\mathcal C}(\mathcal C)$.
\end{defi}

We start with a technical lemma making a link between the formulas
involved in Theorems
\ref{thmcurve} and \ref{thmsurface}.
Inspired by the three dimensional case, let us define the following quantities:
$\alpha_{{\mathbf{S}_0},G}:=\Delta_{\mathbf{S}_0}G$,  
$\beta_{{\mathbf{S}_0},G}:=-2 \left[
  \hess G(\mathbf{S}_0, \boldsymbol{\sigma}_{\mathbf{S}_0,G})
+(\Delta_{\mathbf{S}_0}G)^2(G_{xx}+G_{yy})\right]$
and
$\gamma_{{\mathbf{S}_0},G}:=
- \frac{4\Delta_{\mathbf{S}_0}G}{(d-1)^2}N_{\mathbf{S}_0}\, H_G$.
\begin{lem}\label{petitlemme}
Let $S_0\in\mathbb P^2$ and $\mathcal C=V(G)\subset \mathbb P^2$
be an irreducible
algebraic curve with $G\in\mathbb C[x,y,t]$ being homogeneous of degree $d\ge 2$.
We have $\beta_{{\mathbf{S}_0},G}=\frac{2}{(d-1)^2}N_{\mathbf{S}_0}\, H_G 
        =\frac{\gamma_{{\mathbf{S}_0},G}}{2 \alpha_{{\mathbf{S}_0},G}}$
and so
$\Phi_{\mathbf{S}_0,G}(\mathbf{m})=\lambda_0\cdot \mathbf{m}
+\lambda_1\cdot\boldsymbol{\sigma}
_{\mathbf{S}_0,G}(\mathbf{m}), $
with $[\lambda_0:\lambda_1]=[-\beta_{{\mathbf{S}_0},G}:
\alpha_{{\mathbf{S}_0},G}](\mathbf{m})$ (i.e. $\alpha_{{\mathbf{S}_0},G}(\mathbf{m})\lambda_0+\beta
_{{\mathbf{S}_0},G}(\mathbf{m})\lambda_1=0$ if 
$(\alpha_{{\mathbf{S}_0},G},\beta_{{\mathbf{S}_0},G})(\mathbf{m})\ne
\mathbf{0}$).
\end{lem}
We omit the straightforward proof of this lemma.
\subsection{Caustic of a surface from a light position on a revolution axis}\label{revo}
To simplify, we consider
the case of a surface $\mathcal Z=V(F)$ with axis of revolution $V(x,y)$.
We will use the fact that $F=G\circ h$ where $h(x,y,z,t)=(\sqrt{x^2+y^2},z,t)$ for some homogeneous polynomial
$G\in\mathbb C[r,z,t]$ with
monomials of even degree in $r$.
Such a surface $\mathcal Z$ is written $\mathcal{R}(G)$ and 
is called surface of revolution
of axis $V(x,y)$ of the curve $V(G)\subset\mathbb P^2$.
\begin{thm}\label{thmrevolution}
Let $\mathcal Z=\mathcal R(G)$ with $G\in\mathbb C[r,z,t]$ irreducible
homogeneous of degree $d\ge 2$ (the monomials of $G$ being of even degree in $r$). Assume that $\mathcal Z\not\subseteq V(\Delta_{\mathbf{S}}F,
 (F_x^2+F_y^2+F_z^2)\hess F(\mathbf{S},\mathbf{S}))$.
Let $S[0:0:z_0:t_0]\in\mathbb P^3$ and $S_0[0:z_0:t_0]\in\mathbb P^2$.

If $d=2$ and if $S$ is a focal point of $\mathcal Z$,
then $\Sigma_{S}(\mathcal Z)$ is reduced to another other focal point.
Otherwise, we have
$\Sigma_{S}(\mathcal Z)=V(x,y)\cup \mathcal{R}(\Sigma_{S_0}
(V(G)))$.
\end{thm}
\begin{proof}
We have
\begin{equation}\label{alpharev}
\alpha=[z_0 F_z+t_0F_t]
=\Delta_{\mathbf S_0}G\circ h.
\end{equation}
Now let us prove that
\begin{equation}\label{gammarev}
\gamma=\left[\frac{G_r}r
\gamma_{\mathbf S_0,G}\right]\circ h,
\end{equation}
Since $\Delta_{\mathbf S}F=\Delta_{\mathbf S_0}G\circ h$ and $N_{\mathbf S}=N_{\mathbf S_0}\circ h$, we just have to prove that 
$H_F=\left[\frac{G_r}rH_G\right]\circ h$ on $\mathcal Z$.
Recall that $\frac{t^2}{(d-1)^2}H_F=h_F$ on $\mathcal Z$ and that
$\frac{t^2}{(d-1)^2}H_G=h_G$ on $V(G)$ with
$$h_F:= \left|\begin{array}{cccc}F_{xx}&F_{xy}&F_{xz}&F_x\\
F_{xy}&F_{yy}&F_{yz}&F_y\\
F_{xz}&F_{yz}&F_{zz}&F_z\\
F_x&F_y&F_z&0\end{array}
\right|\ \ \mbox{and}\ \
  h_G:= \left|\begin{array}{cccc}G_{rr}&G_{rz}&G_r\\
G_{rz}&G_{zz}&G_z\\
G_r&G_z&0\end{array}
\right|.$$
We will write as usual $G_r$, $G_z$, $G_t$ for the first order derivatives
of $G$ and $G_{rr}$, $G_{rz}$, $G_{rt}$, $G_{zt}$ and $G_{tt}$ for
the second order derivatives. We also write 
$F_{r}:=G_r\circ h$ and we define analogously
$F_{xr}$, $F_{yr}$, $F_{zr}$ and $F_{rr}$. 
Due to the particular form of $F$, we immediately obtain that
$$F_x=\frac xr F_r,\ \ F_y=\frac yr F_r,\ \ F_{xt}=\frac xr F_{rt},
\ \ F_{yt}=\frac yr F_{rt}, $$
$$F_{xx}=\frac {F_r}r+\frac{x^2}{r^2}F_{rr}-\frac{x^2}{r^3}F_r,\ \ 
   F_{yy}=\frac {F_r}r+\frac{y^2}{r^2}F_{rr}-\frac{y^2}{r^3}F_r,\ \ $$
$$F_{xy}=\frac{xy}{r^2}F_{rr}-\frac{xy}{r^3}F_r,\ \ 
F_{xz}=\frac xr F_{rz},\ \ F_{yz}=\frac yr F_{rz}.$$
with $r:=\sqrt{x^2+y^2}$.
Due to these relations and to the above formula
of $h_F$, we have
\footnote{writing respectively $L_i$ and $C_i$ for the i-th line and for the i-th row, we make successively the following linear changes: $L_1\leftarrow
yL_1-xL_2$, $C_1\leftarrow yC_1-xC_2$ and $L_2\leftarrow rL_2+\frac xr L_1$}
$$h_F= \frac 1{r y^2}\left|\begin{array}{cccc}rF_r&-\frac{xF_r}r&0&0\\
0&\frac{y^2}rF_{rr}&yF_{rz}&yF_r\\
0&\frac y rF_{rz}&F_{zz}&F_z\\
0&\frac yrF_r&F_z&0\end{array}
\right|=\frac {F_r}r\left|\begin{array}{cccc}F_{rr}&F_{rz}&F_r\\
F_{rz}&F_{zz}&F_z\\
F_r&F_z&0\end{array}
\right|$$
and (\ref{gammarev}) follows. Now let us prove that
\begin{equation}\label{betarev}
\beta=\left[\beta_{\mathbf S_0,G}-
      2(\Delta_{\mathbf S_0}G)^2\frac{G_r}r\right]\circ h.
\end{equation}
Using the above formulas, we obtain
$F_{xx}+F_{yy}+F_{zz}=F_{rr}+F_{zz}+\frac{F_r}r$.
Now (\ref{betarev}) comes from the definition of $\beta$,
and from the above expressions of $F_x$, $F_y$, $F_{xz}$, $F_{yz}$ $F_{zz}$, $F_{xt}$, $F_{yt}$ and $F_{tt}$.
Let $m[x:y:z:t]\in\mathcal Z\setminus V(\alpha,
\beta)$.
To prove the result, it is enough to prove that the two solutions
$[\lambda_0:\lambda_1]$ of (\ref{formequadratique}) are 
$[\lambda_0^{(1)}:\lambda_1^{(1)}]=[2\Delta_{\mathbf S_0}G.G_r:r]\circ h$
and
$[\lambda_0^{(2)}:\lambda_1^{(2)}]=[-2h_G N_{\mathbf S_0}:t^2
     \Delta_{\mathbf S_0}G]\circ h =\Phi_{\mathbf S_0,G}\circ h$.
Indeed, since $r=\sqrt{x^2+y^2}$, $x_0=y_0=0$, $F_x=\frac x r F_r$
and $F_y=\frac y r F_r$, we have
$\mathbf M_1:=\lambda_0^{(1)}\cdot\Idd+
\lambda_1^{(1)}\cdot\boldsymbol{\sigma}\in
V(x,y) $
and
$\mathbf M_2:=\lambda_0^{(2)}\cdot\Idd+
\lambda_1^{(2)}\cdot\boldsymbol{\sigma}
=\left[x\frac{R_2}r:y \frac{R_2}r:Z_2:T_2\right]\circ h$
(using Lemma \ref{petitlemme})
with
$\boldsymbol{\Phi}_{\mathbf S_0,G}=(R_2,Z_2,T_2)$.
Observe that $R_2/r\in\mathbb C[r,z,t]$.
Due to Theorem \ref{thmcurve}, the Zariski closure of
$\Phi_{S_0,V(G)}(V(G))$ is
the caustic $\Sigma_{S_0}(V(G))$. 
Observe that, for every $(r,z,t)$, the set $[x:y]$ goes along 
$\mathbb P^1$ when $(x,y)$ moves in $\{(x,y)\in\mathbb C^2\ :\ x^2+y^2=r^2\}$.
So the Zariski closure of $M_2(\mathcal Z)$ 
is the revolution surface
$\mathcal R(\Sigma_{S_0}(V(G)))$.
Observe moreover that the Zariski closure of $M_1(\mathcal Z)$ is $V(x,y)$ unless it is a single point $A$ of $V(x,y)$, 
which would mean that every reflected
line contains this point $A$, this would imply that $\Sigma_{S_0}(V(G))$ is reduced to a point and so that $V(G)$ is a conic and $S_0$
one of its focal point (see e.g. \cite{fredsoaz2}).
To prove that $[\lambda_0^{(1)}:\lambda_1^{(1)}]$
and $[\lambda_0^{(2)}:\lambda_1^{(2)}]$ are the solutions of 
(\ref{formequadratique}), it is enough to prove that 
\begin{equation}\label{somme}
[\gamma:\alpha]
=[\lambda_0^{(1)}\lambda_0^{(2)}:\lambda_1^{(1)}\lambda_1^{(2)}]
\end{equation}
\begin{equation}\label{produit}
\mbox{and}\quad[-\beta:\alpha]
 =[\lambda_0^{(1)}\lambda_1^{(2)}+\lambda_0^{(2)}\lambda_1^{(1)}:\lambda_1^{(1)}\lambda_1^{(2)}].
\end{equation}
Now (\ref{somme}) comes from $[\lambda_0^{(1)}\lambda_0^{(2)}:\lambda_1^{(1)}\lambda_1^{(2)}]=[-4h_GN_{\mathbf S_0}G_r:rt^2]\circ h$, (\ref{alpharev}) and (\ref{gammarev}).
Now  (\ref{produit})
is equivalent to 
$\beta=\left[-2(\Delta_{\mathbf S_0}G)^2
   \frac{G_r}r+\frac{2h_GN_{\mathbf S_0}}{t^2}\right]\circ h $. So (\ref{produit}) comes from (\ref{betarev}) and Lemma \ref{petitlemme}. 
\end{proof}
\begin{rqe}
Due to Theorem \ref{thmrevolution}, the caustic by reflection of a sphere $\mathcal S$ of center $A$ from
$S\ne A$ is the union of the line $(A\ S)$ and of the revolution surface
of axis $(A\ S)$ obtained from the caustic curve of the circle 
$\mathcal S\cap \mathcal P$ where $\mathcal P$ is any plane containing 
$(A\ S)$.
\end{rqe}
We consider the case where 
$\mathcal Z=V(x^2+y^2-2zt)$ with $S=[0:0:z_0:1]$
(with $S_0[0:z_0:1]$) with 
$z_0\ne 1/2$ (i.e. $S\in V(x,y)$ and $S$ is not a focal point
of $\mathcal P$). Observe that $\mathcal Z=\mathcal R(V(G))$
with $G(r,z,t)=r^2/2-zt$. Due to \cite{fredsoaz1}, 
$\Sigma_{S_0}(V(G))$ has degree 6 except if $z_0=0$
(corresponding to $S\in\mathcal Z$) and, in this
last case, $\Sigma_{S_0}(V(G))$ has degree 4.
More precisely:
\begin{prop}\label{parabaxe}
Let $\mathcal Z=V(x^2+y^2-2zt)\subset\mathbb P^3$ and $S[0:0:z_0:1]\in\mathbb P^3$
with $z_0\ne 1/2$.
Then $\Sigma_S(\mathcal Z)=V(x,y)\cup \mathcal R(V(H))$, where
\begin{itemize}
\item if $z_0\ne 0$, the curve $V(H)$ is the sextic given by
\begin{multline*}
H(r,z,t):=27\, r^4z^2-512\, z^3t^3+288\, z^2r^2t^2+
108\, r^4t^2z_0^4+\\
(3072\, zt^5-24r^4zt-512t^6-6144\ z^2t^4+4992\, zr^2t^3+4096\, z^3t^3-1536\, 
z^2r^2t^2-2112\, r^2t^4-1068\, r^4t^2-8\, r^6)z_0^3\\
+(-1536\, zt^5-10560\, zr^2t^3-6144\, z^3t^3+6144\, z^2t^4+288\, r^2t^4+108\, r^4z^2-168\, r^4zt+3195\, r^4t^2+72\, r^6+2688\, z^2r^
2t^2)z_0^2\\
+(-1536\, z^2t^4+3072\, z^3t^3+90\, r^4zt-108\, r^4z^2-1728\, r^4t^2-1536\, z^2r^2t^2+4032\, zr^2t^3-162\, r^6)z_0.
\end{multline*}
\item if $z_0=0$, the curve $V(H)$ is the cuartic given by
$H(r,z,t):=27\, r^4-512\, zt^3+288\, r^2t^2$.
\end{itemize}
\end{prop}
\begin{figure}[!ht]
\centering
\makebox{
\includegraphics[scale=.7]{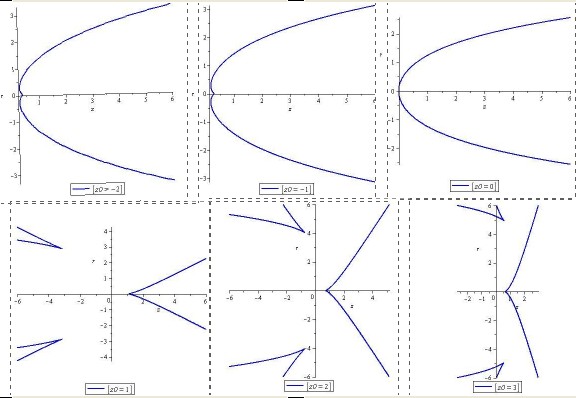}}
\caption{\label{fig:paraboloide1} {Caustics $\mathcal {C'}$ of $V(r^2-2zt)$ from $[0:z_0:1]$
for $z_0=-1,-2, 0,  1,  2, 3$. For every $z_0$, 
the two-dimensional part of the caustic by reflection of $V(x^2+y^2-2zt)$ from 
$[0:0:z_0:1]$ is the revolution surface of $\mathcal {C'}$
around $V(x,y)$.}}
\end{figure}
\begin{proof}
We have $\mathcal Z=\mathcal R(V(G))$ with $G(r,z,t)=(r^2-2zt)/2$.
Due to Theorem \ref{thmrevolution}, we know that 
$\Sigma_S(\mathcal Z)=V(x,y)\cup\mathcal R(\Sigma_{S_0}(\mathcal C)) $
with $\mathcal C=V(G)\subseteq\mathbb P^2$ and $S_0[0:z_0:1]\in\mathbb P^2$.
Observe that $\mathcal C$ admits the rational parametrization $(u,v)\mapsto [uv:(u^2/2):v^2]$. Due to Theorem
\ref{thmcurve}, $\Sigma_{S_0}(\mathcal C)$ has the
rational parametrization
$ (u,v)\mapsto \Phi_{\mathbf {S}_0,G}([uv:(u^2/2):v^2])$ and
\begin{multline*}
\boldsymbol{\Phi}_{\mathbf{S}_0,G}\left(uv,\frac{u^2}2,v^2\right)
  =\left(2u^3v^3(1-2z_0),\frac 14(4z_0^2v^6+6z_0v^2u^2(v^2-u^2)
    +u^4(u^2+6v^2)),\right.\\
  \left.\frac 12(2z_0-1)v^4(2z_0v^2-3u^2)\right),
\end{multline*}
which parametrizes $\Sigma_{S_0}(\mathcal C)$.
So $\Sigma_{S_0}(\mathcal C)=V(H)$.
\end{proof}
\subsection{Caustic of a cylinder}\label{cyli}
To simplify, we restrict ourselves to the study of a cylindrical 
surface $\mathcal Z=V(F)$ with axis $V(x,y)$.
We will use the fact that $F(x,y,z,t)=G(x,y,t)$ for 
some homogeneous polynomial $G\in\mathbb C[x,y,t]$.
Such a surface $\mathcal Z$ is called the cylinder of axis $V(x,y)$
and of basis $V(G)\subset\mathbb P^2$. We then write $\mathcal Z=Cyl(G)$.
Observe that, in this particular case, the tangent plane to $\mathcal Z$
at $m=[x:y:z:t]$ does not depend on $z$.
\begin{rqe}
If $\mathcal Z=Cyl(G)$ (with $G$ as above) and if
$S[0:0:1:0]\in\mathbb P^3$, then $\mathcal Z\subseteq V(\Delta_{\mathbf{S}}F,
 (F_x^2+F_y^2+F_z^2)\hess F(\mathbf{S},\mathbf{S}))$.
\end{rqe}
\begin{thm}\label{thmcylinder}
Let $S[x_0:y_0:z_0:t_0]\in\mathbb P^3\setminus\{[0:0:1:0]\}$ and let $\mathcal Z=Cyl(G)$
with $G\in\mathbb C[x,y,t]$ an irreducible homogeneous polynomial
of degree $d\ge 2$. Assume that $\mathcal Z\not\subseteq V(\Delta_{\mathbf{S}}F,
 (F_x^2+F_y^2+F_z^2)\hess F(\mathbf{S},\mathbf{S}))$.
We set $S_0[x_0:y_0:t_0]\in\mathbb P^2$.

If $V(G)\not\subseteq V(H_G,N_{\mathbf S_0})$, then
$\Sigma_S(\mathcal Z)=\overline{\sigma(\mathcal Z)}
\cup Cyl(\Sigma_{S_0}(V(G)))$, where $\overline{\sigma(\mathcal Z)}$ 
is the algebraic curve 
corresponding to the Zariski closure of the sets of
orthogonal symmetrics of $S$ 
with respect to the tangent planes to $\mathcal Z$.
Otherwise  
$\Sigma_S(\mathcal Z)=\overline{\sigma(\mathcal Z)}$.
\end{thm}
\begin{proof}
Observe that, since $F_{xz}=F_{yz}=F_{zz}=F_{zt}=0$, we have $H_F=0$
and so ${\gamma}=0$. 
Let $m[x:y:z:t]\in\mathcal Z\setminus V(\alpha,
\beta)$. We have
${\alpha}=\Delta_{\mathbf S}F=
\Delta_{\mathbf S_0}G\circ h$ and 
${\beta}=
{\beta}_{\mathbf S_0,G}\circ h$
with $h(x,y,z,t)=(x,y,t)$. So (\ref{formequadratique}) becomes
$\lambda_0(\alpha(\mathbf{m})\lambda_0+
\beta(\mathbf{m})\lambda_1)=0$ and
its solutions $[\lambda_0:\lambda_1]\in\mathbb P^1$ are
$[0:1]$ and $[-\beta(\mathbf{m}):
\alpha(\mathbf{m})]$.
The corresponding points on 
$\Sigma_S(\mathcal Z)$ are
$M_1(m):=\sigma(m)$
and $M_2(m)[X_2(m):Y_2(m):Z_2(m):T_2(m)]$
with $[X_2:Y_2:T_2]=\Phi_{S_0,V(G)}$ (due to 
Lemma \ref{petitlemme}) and
$Z_2(m)= \left [-\frac{2H_G(h(m))N_{\mathbf S_0}(h(m))}{(d-1)^{2}}z
    +\Delta_{\mathbf S_0}G(h(m))(G_x^2(h(m))+G_y^2(h(m))z_0\right]$.
Due to Theorem \ref{thmcurve}, the Zariski closure of $\Phi_{S_0,V(G)}
(V(G))$ is $\Sigma_{S_0}(V(G))$.
If $V(G)\not\subseteq V(H_GN_{\mathbf S_0})$, then, for every $[x:y:t]\in
V(G)$, $Z_2(x,y,z,t)$ goes all over $\mathbb C$ when $z$ describes $\mathbb C$.
If $V(G)\subseteq V(H_GN_{\mathbf S_0})$, then, due
to Lemma \ref{petitlemme},
$\boldsymbol{\beta}=0$ on $\mathcal Z$, which implies
that $M_2=M_1$ on $\mathcal Z$.
\end{proof}
\begin{prop}[parabolic cylinder with light at infinity]\label{parabcyl}
The caustic of reflection of $\mathcal Z=V(y^2-2xt)\subset\mathbb P^3$
from $S[1-v^2:2v:z_0:0]$ with $v\ne 0$ is $\overline{\sigma
 (\mathcal Z)}\cup V(H)$, with
\begin{multline*}
H(x,y,z,t)=
4y^{3}(1-v^6)+(-27t^3+108xt^2-72y^{2}t+24xy^{2}-108x^{2}t)(v+v^5)\\
+12y(2x+3t+y)(2x+3t-y)(v^2-v^{4})\\
+2(16x^{3}+27t^3-36x^{2}t-24xy^{2}+216xt^2-144y^{2}t)v^{3},
\end{multline*}
(geometrically, $v$ corresponds to the
tangent of the half-angle of $(1,0)$ with the direction of $S$).
\end{prop}
\begin{figure}[!ht]
\centering
\makebox{
\includegraphics[scale=.3]{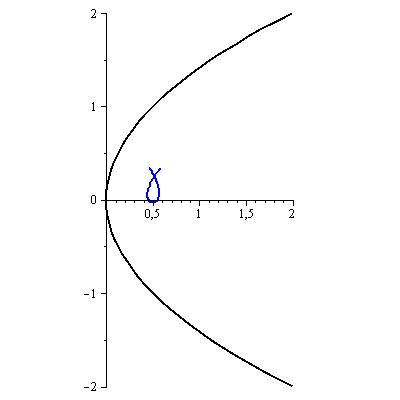}
\includegraphics[scale=.3]{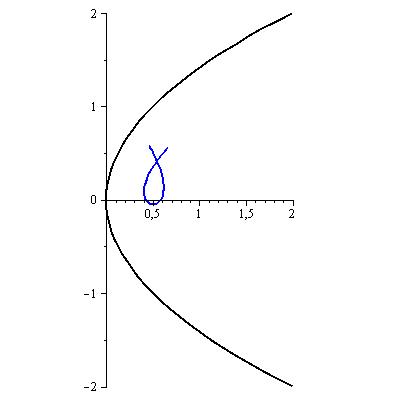}
\includegraphics[scale=.3]{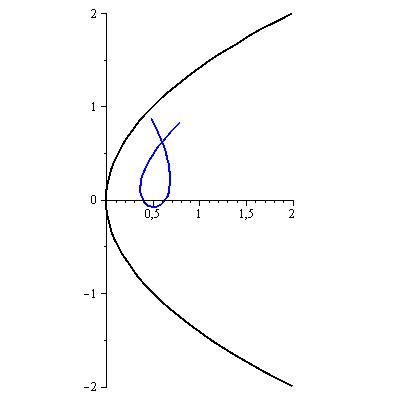}
\includegraphics[scale=.3]{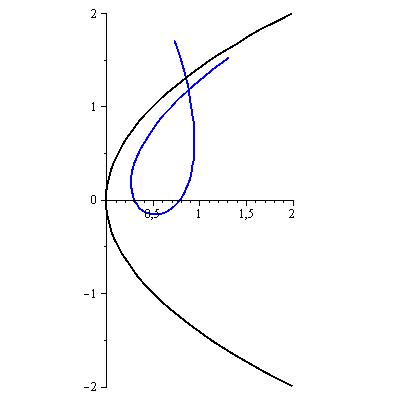}}\\
\caption{\label{fig:presfoyer} {Caustics of $V(H)$ for $t_0=0$ and $x_0+iy_0=e^{i\theta}$ for
$\theta=\frac\pi {50}, \, \frac\pi {30},\, \frac\pi {20}, \, \frac\pi {10}$.} }
\end{figure}
Such surfaces are used in practice to 
concentrate sunrays on a tube (put along the line made of the focal points
of the parabols) in order to heat the water circulating 
in it. Figure \ref{fig:presfoyer} is a transverse representation of this
solar heater, the tube being at the focal point $(1/2,1)$.
\begin{proof}
Let $S_0[1-v^2:2v:0]$ and $\mathcal C:=V(G)\subset \mathbb P^2$
with $G(x,y):=(y^2-2xt)/2$.
Due to Theorem \ref{thmcylinder}, the caustic by reflection of 
$V(y^2-2xt)\subset\mathbb P^3$ from $S$ is $V(x,y)\cup Cyl(\Sigma_{S_0}(\mathcal C))$.
Moreover we know from \cite{fredsoaz1} that $\deg \Sigma_{S_0}(\mathcal C)=3$.
Using the parametrization $\psi:(a,b)\mapsto \left[\frac{a^2}2:ab:b^2\right]$
of $\mathcal C$ together with Theorem \ref{thmcurve}, we conclude
that $\Phi_{S_0,V(G)}\circ\psi(a,b)=[X_1(a,b):Y_1(a,b):Z_1(a,b)]$ 
is a parametrization of $\Sigma_{S'}(\mathcal C)$.
We obtain $\boldsymbol{\sigma}_{\mathbf{S}_0,G}=((y^2-t^2)x_0+2ty_0y,
(t^2-y^2)y_0+2tx_0y$ and so
$$X_1=2v(1-v^2)a^3b^3+12v^2a^2b^4-6v(1-v^2)ab^5 +(1-v^2)^2b^6$$
$$Y_1=-4v^2a^3b^3+6v(1-v^2)a^2b^4+12v^2ab^5-2v(1-v^2)b^6 ,\quad Z_1=2(1+v^2)^2b^6.$$
\end{proof}
To complete the study of this example, let us specify 
$\overline{\sigma(\mathcal Z)}$.
\begin{prop}
Under assumptions of the previous result, we have
\begin{itemize}
\item if $v^2\ne -1$, then $\overline{\sigma(\mathcal Z)}
=\mathcal H^\infty\cap V((x^2+y^2)z_0^2-(v^2+1)^2z^2)$;
\item if $z_0\ne 0$ and $v^2=-1$, then $\overline{\sigma(\mathcal Z)}=\mathcal H^\infty\cap V(y+2x)$,
\item if $z_0= 0$ and $v^2=-1$, then $\overline{\sigma(\mathcal Z)}=\{[2:-v:0:0]\}$.
\end{itemize}
\end{prop}
\begin{proof}
We have 
$\boldsymbol{\sigma}(x,a,z,b)
=\left(\begin{array}{c}
  a^2(1-v^2)+4vab-b^2(1-v^2)=:g_1(a,b)\\a^2(-2v)+2(1-v^2)ab+b^22v=:g_2(a,b)\\
(a^2+b^2)z_0=:g_3(a,b)\\0\end{array}\right).
$

First observe that
$\overline{\sigma(\mathcal Z)}$ is included in $\mathcal H^\infty$.
Second, we compute the resultant in $a$ of $(xg_3(a,1)-g_1(a,1),yg_3(a,1)-g_2(a,1))$ and obtain $4(1+v^2)^2((x^2+y^2)z_0-(v^2+1)^2)$.
If $v^2\ne -1$, this resultant gives the result (by homogeneization with $z$).

Assume now that $v^2=1$. We have
$\boldsymbol{\sigma}(x,a,z,b)=(a+vb)\left(\begin{array}{c}
  2(a+vb)\\-v(a+vb)\\(a-vb)z_0\\0\end{array}\right).$

If $z_0=0$, we have $\boldsymbol{\sigma}(x,a,z,b)=(a+vb)^2(2,-v,0,0)$ and so
$\overline{\sigma(\mathcal Z)}=\{[2:-v:0:0]\}$.
Finally, if $z_0\ne 0$ (still with $v^2=1$), using the fact that $[a-vb:a+vb]$
describes $\mathbb P^1$ when $[a:b]$ moves in $\mathbb P^1$, we obtain
$\overline{\sigma(\mathcal Z)}=\mathcal H^\infty\cap V(y+2x)$.
\end{proof}

\end{appendix}

\end{document}